\documentclass[final]{siamltex}

\usepackage{amsfonts}

\usepackage{graphicx}

\usepackage{amsmath}

\newtheorem{assumption}{Assumption}

\title{Galerkin approximations
for the stochastic Burgers equation\thanks{
This work has been supported by 
the Collaborative Research Centre~701
``Spectral Structures and 
Topological Methods in Mathematics'', 
by the research project
``Mehrskalenanalyse stochastischer 
partieller Differentialgleichungen 
(SPDEs)'' and by
the research project ``Numerical 
solutions of stochastic
differential equations with
non-globally Lipschitz continuous
coefficients'' (all funded by the
German Research Foundation).}}

\author{Dirk
Bl\"{o}mker\thanks{Institut f\"ur Mathematik,
Universit\"at Augsburg, 
86135~Augsburg, Germany,
{\tt dirk.bloemker@math.uni-augsburg.de}}
\and
Arnulf Jentzen\thanks{Seminar for Applied Mathematics, 
Swiss Federal Institute of Technology,
8092 Zurich, 
Switzerland, 
{\tt arnulf.jentzen@sam.math.ethz.ch}}
}

\begin{document}

\maketitle

\begin{abstract}
Existence and uniqueness for
semilinear stochastic
evolution equations with additive
noise by means of finite dimensional
Galerkin approximations is established
and the convergence
rate of the Galerkin approximations
to the solution of the stochastic evolution equation
is estimated.

These abstract results are
applied to several examples
of stochastic partial differential
equations (SPDEs) of evolutionary
type including a stochastic heat equation, 
a stochastic reaction diffusion equation
and a stochastic Burgers equation.
The estimated convergence rates 
are illustrated by 
numerical simulations.

The main novelty in this article is
to estimate the difference of
the finite dimensional Galerkin approximations
and of the solution of the infinite dimensional SPDE
uniformly in space, i.e.,~in
the $L^\infty$-topology, instead
of the usual Hilbert space estimates
in the $L^2$-topology,
that were shown before.
\end{abstract}


\pagestyle{myheadings}
\thispagestyle{plain}
\markboth{D. BL\"{O}MKER AND A. JENTZEN}{GALERKIN APPROXIMATIONS FOR SPDES}

\section{Introduction}

In this work we present a 
general 
abstract 
result for the spatial approximation 
of stochastic
evolution equations with additive noise
via Galerkin methods.
This abstract result is applied to
several examples of stochastic
partial differential equations (SPDEs)
of evolutionary type including 
a stochastic heat equation, 
a stochastic reaction diffusion 
equation and a stochastic Burgers 
equation.
In all examples we need to  verify  the following conditions.
First, we need the rate of 
approximation of the linear 
equation obtained by omitting the nonlinear
term in the stochastic evolution equation.
Then one needs a quite weak 
Lipschitz condition for the 
nonlinearity and finally 
a uniform bound on the sequence of approximations.
These results are the key for the main theorem
(see Theorem~\ref{mainthm}).
The main novelty in this article is
to estimate the difference of
the finite dimensional Galerkin approximations
and of the solution of the infinite dimensional SPDE
uniformly in space, i.e., in
the $L^\infty$-topology, instead
of the usual Hilbert space estimates shown before
in the $L^2$-topology.

Although there are several different methods 
using  finite dimensional approximations like, for instance, spectral Galerkin,
finite elements, or wavelets,
we focus here on the spectral Galerkin 
method.
Thus the finite dimensional approximations are given by an expansion in terms 
of the eigenfunctions of a dominant linear operator.
This spectral Galerkin method is one 
of the key tools 
in the analysis of stochastic or 
deterministic PDEs. 
For SPDEs see, for example, 
\cite{FF-DG:95, GdP-AD:96, FF-MR:08, DB-FF-MR:09}, 
where the Galerkin method 
was used to  establish the 
existence of solutions.
Moreover, spectral methods are an effective tool for 
numerical simulations, especially on domains, like the interval,
where fast Fourier-transforms are available.
Nevertheless, it is limited on domains, 
where the eigenfunctions of the dominant linear operator 
are not explicitly known.
In recent years there has also been 
a significant interest in  
analytic results for the rate of approximation
using a spectral Galerkin method 
as a numerical method
for SPDEs;
see, for example,
\cite{gk96,ks01}
for SPDEs with 
one-dimensional possibly
non-additive noise
and globally Lipschitz
continuous nonlinearities,
\cite{lr04,ls06, mrw07, mrw08, jk09b, klns11}
for SPDEs with
possibly infinite dimensional 
additive noise and
globally Lipschitz continuous
nonlinearities,
\cite{l03b,AJ:09}
for SPDEs with
possibly infinite dimensional 
additive noise and
non-globally Lipschitz continuous
nonlinearities,
and 
\cite{h02,h03a,mr07a, mr07b}
for SPDEs with
possibly infinite dimensional
non-additive noise
and globally
Lipschitz continuous
nonlinearities.
In most of the above named
references also the full discretization 
is treated including the 
time discretization.

In order to illustrate the main
result of this article we
limit ourself in this introductory 
section to a stochastic Burgers
equation with Dirichlet boundary
conditions and refer to
Section~\ref{secmain} for the 
general result and 
to Section~\ref{secex} 
for further examples.
To this end let $ T \in (0,\infty) $
be a real number,
let 
$
  \left( \Omega, \mathcal{F}, \mathbb{P}
  \right) 
$ 
be a given probability space
and let 
$ 
  X \colon [0,T] \times \Omega
  \rightarrow 
  C( [0,1], \mathbb{R}
  ) 
$ 
be the up to indistinguishability
unique solution 
process
of the SPDE
\begin{equation}
\label{burgerEq}
  dX_t(x) =
  \left[
    \frac{ \partial^2 }{
      \partial x^2
    } X_t(x)
    -
    X_t(x) \cdot
    \frac{ \partial }{ \partial x } X_t(x)
  \right] 
  dt
  +
  dW_t(x),
\;
  X_t(0) = X_t(1) = 0,
\;
  X_0 = 0
\end{equation}
for $ t \in [0,T] $ and
$ x \in (0,1)$, where 
$ W_t $, $ t \in [0,T] $,
is a cylindrical $ I $-Wiener 
process on
$ L^2( (0,1), \mathbb{R} ) $,
which models space-time white noise on $(0,1)$.
In this introductory section
the initial value $ X_0=0 $ is
zero for simplicity of presentation
and we refer to 
Section~\ref{stochburger}
below
for a more general stochastic
Burgers equation with a possibly
non-zero initial value.
The existence and uniqueness of solutions 
of the stochastic Burgers equation 
was, e.g., 
studied in Da Prato
\& Gatarek~\cite{dg95} for
colored noise
and in Da Prato, 
Debussche \& 
Temam~\cite{GdP-AD-RT:94} 
for space-time white noise
(see also Chapter~14 in Da Prato
and Zabczyk~\cite{dz96}).

Recently, Alabert \& Gy\"ongy showed 
the following error estimate 
for spatial discretizations 
in the $L^2$-topology
(see Theorem 2.2 in \cite{ag06}): 
\begin{equation}\label{gyongyest}
  \mathbb{P}\!\left[
    \sup_{0 \leq t \leq T}
    \Big(
    \int_0^1
    |X_t(x) - X_t^N(x)|^2 \, 
    dx
    \Big)^{1/2}
    \leq
    C_{\varepsilon} 
    \cdot 
    N^{  \varepsilon - \frac{1}{2} 
    }
  \right] = 1
\end{equation}
for every $N \in \mathbb{N}:=
\left\{1,2,\dots\right\}$ and
every arbitrarily small
$ \varepsilon \in (0,\frac{1}{2}) $
with random variables 
$ 
  C_\varepsilon \colon \Omega
  \rightarrow [0,\infty)
$, 
$ \varepsilon \in (0,\frac{1}{2})$, 
where the 
$ X^N $, $ N \in \mathbb{N} $, 
are given by finite 
differences approximations.
Our results
(see Lemma~\ref{constructOO},
Theorem~\ref{mainthm}
and Lemma~\ref{BurgerSol})
yield the following estimate
for the stochastic Burgers 
equation~\eqref{burgerEq}
(see Section \ref{stochburger}):
\begin{equation}
\label{illustratethm1}
  \mathbb{P}\!\left[
    \sup_{0 \leq t \leq T}
    \sup_{0 \leq x \leq 1}
    \left|
      X_t(x) - X_t^N(x)
    \right|
    \leq
    C_{\varepsilon}  
    \cdot
    N^{\varepsilon - \frac{1}{2}
    }
  \right] = 1
\end{equation}
for every $ N \in \mathbb{N}$ and
every arbitrarily small
$ \varepsilon \in (0,\frac{1}{2}) $
with random variables 
$ 
  C_\varepsilon \colon \Omega 
  \rightarrow [0,\infty)
$, 
$ 
  \varepsilon \in (0,\frac{1}{2})
$, 
where 
$X^N$, $ N \in \mathbb{N} $, 
are spectral Galerkin approximations. 
Thus, although the spatial error criteria 
is estimated in the bigger $L^\infty$-norm
instead of the $L^2$-norm,
the convergence rate remains $ \frac{1}{2}- $.
This convergence rate with respect to
the strong $L^\infty$-norm
is also corroborated by a numerical example
(see Section \ref{secex}).
(For a real number $ a \in (0,\infty) $, 
we write $ a- $ for the convergence
order, if the convergence order is higher
than $ a - \varepsilon $ for every
arbitrarily small $ \varepsilon \in (0,a) $.)

A further instructive related result 
is given by Liu \cite{l03b}. 
He treats stochastic
reaction diffusion equations of the 
Ginzburg-Landau type which fit 
in the abstract setting in Section~\ref{sec2}.
For such equations he obtained estimates
in the $ H^r $-topology with the rate
$ ( \frac{ 1 }{ 2 } - r ) - $
for every $ r \in (0,\frac{1}{2}) $.
The convergence rates he obtained in the $ H^r $-topologoy 
with $ r \in ( 0, \frac{ 1 }{ 2 } ) $
can, in general, not be improved and,
by using Sobolev embeddings, 
his bounds also yield estimates in 
the $ L^p $-topology with $ p \in (2,\infty) $.
Nevertheless, such estimates do
not yield convergence in 
the $ L^{ \infty } $-topology,
since in one dimension
$ H^r $ is embedded into $L^\infty$
for $ r > \frac{1}{2} $ only.
Moreover, in contrast to 
\eqref{illustratethm1} this would not 
give a convergence 
rate 
$ \frac12- $ 
in any 
$ L^p $-topology
where
$ 
  p \in (2,\infty]
$.
%
%
%
%
%
%
%
%

The rest of the paper is organized 
as follows. 
Section \ref{sec2} gives 
the setting and the assumptions 
for the main result,
which is then presented in Section \ref{secmain}.
In Section \ref{secex} we discuss our examples, while in the final section most of the proofs are stated.

Next we add
that after the preprint
version \cite{bj09}
of this article has appeared,
a number of related
results appeared in
the literature;
see, e.g.,
\cite{BrzezniakCarelliProhl2010,HairerVoss2010,k11,CoxVanNeerven2012,CoxHausenblas2012,Doersek2012,DBKam}.
In particular, we mention
\cite{CoxVanNeerven2012,CoxHausenblas2012} for
temporal and spatial
discretization estimates
in Banach spaces
that imply estimates
in the $ L^{ \infty } $-norm
as well as
\cite{k11}
for the analysis of
spectral Galerkin methods
for semilinear SPDEs with 
possibly non-additive
noise and globally Lipschitz
continuous nonlinearities.
We also refer, e.g., to
\cite{HairerVoss2010}
for further spatial approximations
of stochastic Burgers equations
and, e.g.,
to
\cite{BrzezniakCarelliProhl2010,Doersek2012}
for the analysis of 
spatial and temporal-spatial discretizations
of stochastic Navier-Stokes
equations.
Finally, we would like to
point out that 
parts of this article 
(see Subsection~\ref{stochheat})
appeared 
in the thesis \cite{j09b}
(see Section~2.2.3 in \cite{j09b}).


\section{Setting and assumptions} \label{sec2}


Throughout this article
suppose that the following
setting and the following
assumptions are fulfilled.

The first assumption is a regularity and 
approximation condition on 
the semigroup of the linear 
operator of the considered SPDE.
The second is an appropriate
Lipschitz condition 
on the nonlinearity
of the considered SPDE.
The third is an assumption on the approximation of the 
stochastic convolution
and the initial value of 
the considered SPDE
while the final one is a 
uniform bound on finite 
dimensional approximations
of the considered SPDE.

Let $ T \in (0,\infty) $,
let 
$( \Omega, \mathcal{F}, \mathbb{P}) $ 
be a probability space
and let $ ( V, \left\| \cdot \right\|_V ) $
and $ ( W, \left\| \cdot \right\|_W ) $ 
be two $ \mathbb{R} $-Banach 
spaces.
Moreover, let 
$ 
  P_N \colon V \rightarrow V 
$, 
$ 
  N \in \mathbb{N}
$,
be a sequence of bounded 
linear operators
from $ V $ to $ V $.

\begin{assumption}[Semigroup $S$] \label{semigroup}
Let
$ \alpha \in [0,1) $ 
and $ \gamma \in (0,\infty) $ 
be real numbers and let
$ 
  S \colon (0,T] \rightarrow L(W,V) 
$
be a strongly continuous mapping 
which satisfies
$
  \sup_{ t \in (0,T] }
  \left(
    t^{ \alpha }
    \left\| S_t \right\|_{L(W,V)} 
  \right)
  < \infty
$
and
$
  \sup_{N \in \mathbb{N}}
  \sup_{ t \in (0,T] }
  \left(
    t^{ \alpha }
    N^{ \gamma }
    \left\| S_t - P_N S_t \right\|_{L(W,V)} 
  \right)
  < \infty 
$.
\end{assumption}

\begin{assumption}[Nonlinearity $F$]\label{drift}
Let 
$ 
  F \colon V \rightarrow W 
$ 
be a mapping which 
satisfies
$
  \sup_{ 
  \substack{
    \| v \|_V,
    \| w \|_V \leq r, \,
    v \neq w 
  } }
  \frac{ 
    \left\| F(v) - F(w)
    \right\|_W 
  }{
    \left\| 
      v - w
    \right\|_V
  } 
  < \infty
$
for every
$ r \in (0,\infty) $.
\end{assumption}

\begin{assumption}[Stochastic
process $O$]\label{stochconv}
Let
$ 
  O \colon [0,T] \times \Omega 
  \rightarrow V 
$
be a stochastic process
with continuous 
sample paths and 
$
  \sup_{ N \in \mathbb{N} }
  \sup_{0 \leq t \leq T}
  N^\gamma 
  \left\| O_t(\omega) - 
  P_N( O_t(\omega) ) \right\|_V
  < \infty
$ for every
$ \omega \in \Omega $,
where $ \gamma \in (0,\infty) $
is given in Assumption \ref{semigroup}.
\end{assumption}

\begin{assumption}[Existence of
solutions]\label{solution}
Let
$ 
  X^N \colon [0,T] \times
  \Omega \rightarrow V 
$,
$ 
  N \in \mathbb{N}
$,
be a sequence of
stochastic processes
with continuous sample paths
and with
\begin{equation}
\label{GalSODEs}
  X^N_t(\omega)
  = 
  \int^t_0
    P_N \, S_{ t - s } \,
    F( X^N_s( \omega ) ) \, 
  ds
  +
  P_N( O_t(\omega) )
\quad
  \text{and}
\quad
  \sup_{M \in \mathbb{N}}
  \sup_{ s \in [0,T] }
  \| X^M_s(\omega) \|_V 
  < \infty 
\end{equation}
for every $ t \in [0,T] $, 
$ \omega \in \Omega $
and every $N \in \mathbb{N}$.
\end{assumption}

As usual, we call here a mapping 
$ 
  Y \colon [0,T] \times \Omega
  \rightarrow V 
$ 
a stochastic
process, 
if for every
$ t \in [0,T] $
the mapping
$
  Y_t \colon
  \Omega \ni 
  \omega \mapsto
  Y_t(\omega) 
  :=
  Y(t,\omega)
  \in V
$
is 
$ 
  \mathcal{F}
$/$
  \mathcal{B}(V)
$-measurable.
Additionally, we say that 
a stochastic
process 
$ 
  Y \colon [0,T]
  \times \Omega
  \rightarrow V 
$ 
has continuous
sample paths, if
for every $ \omega \in \Omega $
the mapping
$
  [0,T] \ni t \mapsto 
  Y_t( \omega ) \in V
$
is continuous.
Furthermore, we say that
a mapping 
$ f \colon (0,T] \to L(W,V) $
is strongly continuous if 
for every $ w \in W $
the mapping
$
  (0,T] \ni t \mapsto
  f(t) w 
  \in V
$
is continuous.
Moreover, note that if
$ 
  Y \colon [0,T] \times
  \Omega \rightarrow V 
$
is a stochastic process
with continuous sample paths,
then Assumptions~\ref{semigroup}
and \ref{drift} ensure 
for every
$ \omega \in \Omega $,
$ t \in (0,T] $
and every
$ N \in \mathbb{N} $
that the mapping
$
  (0,t) \ni s
  \mapsto
  P_N \, S_{ t - s } 
  \, F( Y_s( \omega ) )
  \in V
$
is continuous and therefore,
we obtain
for every
$ \omega \in \Omega $,
$ t \in [0,T] $
and every
$ N \in \mathbb{N} $
that the $ V $-valued Bochner 
integral 
$
  \int^t_0P_N \, 
  S_{ t - s } \, F( Y_s(\omega) ) \, 
  ds
  \in V
$
(see \eqref{GalSODEs}
in Assumption~\ref{solution})
is well defined.

\section{Main result}
\label{secmain}

In this section we state 
the main approximation result, 
which is based on the assumptions
of the previous section.
Its proof is postponed to 
Subsection~\ref{sec:proofthm1}.

\begin{theorem}
\label{mainthm}
Let Assumptions 
\ref{semigroup}-\ref{solution}
be fulfilled.
Then there exists
a unique stochastic process
$ 
  X \colon [0,T] \times \Omega
  \rightarrow V 
$ 
with
continuous sample paths
which fulfills
\begin{equation}
\label{eqspdesol}
  X_t(\omega)
  = 
  \int^t_0
  S_{ t - s }
  \, F( X_s( \omega ) ) \, ds
  +
  O_t(\omega) 
\end{equation}
for every $ t \in [0,T] $ and 
every $ \omega \in \Omega $.
Moreover, there exists
an 
$ 
  \mathcal{F}
$/$
  \mathcal{B}([0,\infty))
$-measurable 
mapping
$ 
  C \colon \Omega \rightarrow [0,\infty)
$
such that
\begin{equation}
\label{eq:mainestimate}
  \sup_{0\leq t \leq T} 
  \left\| 
    X_t(\omega)
    - X^N_t(\omega) 
  \right\|_V
  \leq
  C(\omega) \cdot N^{ - \gamma }
\end{equation}
for every $ N \in \mathbb{N}$
and every $ \omega \in \Omega $
where $ \gamma \in (0,\infty) $
is given in Assumption \ref{semigroup}.
\end{theorem}

Let us add three remarks on 
Theorem~\ref{mainthm}.
First, we would like to point
out that the initial value
of the stochastic
evolution equation~\eqref{eqspdesol}
is incorporated in the
driving stochastic process
$
  O \colon [0,T] \times \Omega
  \rightarrow V 
$ (see also Proposition~\ref{constructO}
below for more details).
Second, we emphasize that the
driving stochastic processes
$
  O \colon [0,T] \times \Omega
  \rightarrow V 
$
is not assumed to be a stochastic 
convolution of the semigroup
and a cylindrical Wiener process.
In particular, the stochastic
evolution equation~\eqref{eqspdesol}
covers SPDEs disturbed by fractional
Brownian motions too.
Third, we would like to point
out that 
Theorem~\ref{mainthm}
yields the existence of an 
$ \mathcal{F} $/$ \mathcal{B}( [0,\infty) ) $-measurable 
mapping
$ 
  C \colon \Omega \rightarrow [0,\infty)
$
such that
\eqref{eq:mainestimate}
holds although
the $ \mathbb{R} $-Banach
space 
$ 
  \left( V, \left\| \cdot \right\|_V \right) 
$
is not assumed to be separable.
The sum and the difference of two
$ \mathcal{F} $/$ \mathcal{B}( V ) $-measurable 
mappings on the possibly non-separable 
$ \mathbb{R} $-Banach
space $ V $ are, in general, not
$ \mathcal{F} $/$ \mathcal{B}( V ) $-measurable
anymore.
Nonetheless, it is possible
to establish the existence
of an
$ \mathcal{F} $/$ \mathcal{B}( [0,\infty) ) $-measurable
mapping
$ 
  C \colon \Omega \rightarrow [0,\infty)
$
such that
\eqref{eq:mainestimate}
holds by exploiting
for every $ t \in [0,T] $
and every $ N \in \mathbb{N} $
that 
the difference
$
  O_t - P_N( O_t )
=
  ( I - P_N ) \, O_t
  \colon \Omega \to V
$
is $ \mathcal{F} $/$ \mathcal{B}( V ) $-measurable
(see \eqref{eq:defR}
in the proof
of Theorem~\ref{mainthm}
for more details).
Note that the composition
of two measurable mappings
is measurable (on
non-separable 
$ \mathbb{R} $-Banach
spaces too).
Finally, we note that
the error constant
$ 
  C \colon \Omega \rightarrow [0,\infty)
$
appearing in 
\eqref{eq:mainestimate}
is described explicitly in
the proof of Theorem~\ref{mainthm}
(see 
definition~\eqref{eq:described_constant}
in the proof of 
Theorem~\ref{mainthm}
for details).


\section{Examples}
\label{secex}

%
This section presents
some examples 
of the setting 
in Section \ref{sec2}.
%
%
%
%
%
\subsection{Stochastic heat equation}
\label{stochheat}


In this subsection an important example
of Assumption \ref{stochconv} is presented.
We consider a linear equation with $F=0$ and 
thus consider only the approximation of the 
Ornstein-Uhlenbeck process $O$.

To this end 
let $ d \in \mathbb{N} $ and
let
$ V = W = C( [0,1]^d, \mathbb{R} ) $ be 
the $\mathbb{R}$-Banach space of
continuous functions from 
$ [0,1]^d $ to $ \mathbb{R} $
equipped with the 
supremum norm
$
  \left\| \cdot 
  \right\|_V 
  =
  \left\| \cdot 
  \right\|_W
  =
  \left\| \cdot 
  \right\|_{
    C([0,1]^d , \mathbb{R})
  }
$.
Moreover, 
consider the continuous 
functions
$
  e_i \colon
  [0,1]^d \rightarrow
  \mathbb{R}
$,
$ i \in \mathbb{N}^d 
$,
and the
real numbers
$
  \lambda_i
$,
$
  i \in \mathbb{N}^d
$,
defined through
\begin{equation}
\label{deflambdai}
  e_i(x) :=
  2^{\frac{d}{2}}
  \sin(i_1 \pi x_1)
  \ldots
  \sin(i_d \pi x_d)
\qquad\text{and}\qquad
  \lambda_i
  := \pi^2
  \big( 
    | i_1 |^2 + \ldots + | i_d |^2 
  \big)
\end{equation}
for all 
$ 
  x = (x_1, \dots, x_d) \in [0,1]^d 
$
and all
$ 
  i = 
  (i_1, \dots, i_d) \in \mathbb{N}^d 
$.
Additionally, suppose that the 
bounded linear operators
$ 
  P_N \colon 
  C([0,1]^d, \mathbb{R}) 
  \rightarrow 
  C([0,1]^d, \mathbb{R})  
$,
$ N \in \mathbb{N} $,
are given by
\begin{equation}\label{PNex}
  ( P_N(v))(x)
  =
  \sum_{ i \in\{ 1, \dots, N \}^d}
  \int_{(0,1)^d} e_i(s) \, v(s) \, ds
  \cdot e_i(x)
\end{equation}
for all 
$ 
  x \in [0,1]^d 
$,
$ 
  v \in C([0,1]^d, \mathbb{R}) 
$ 
and all
$ 
  N \in \mathbb{N} 
$.
The linear operators
$ P_N $, $ N \in \mathbb{N} $,
are projection operators, i.e.,
they satisfy $ P_N( P_N( v ) ) = P_N( v ) $
for all $ v \in C( [0,1]^d, \mathbb{R} ) $
and all $ N \in \mathbb{N} $
and their images are the finite
dimensional $ \mathbb{R} $-vector 
spaces
$
  P_N\big( 
    C( [0,1]^d, \mathbb{R} )
  \big)
$,
$ N \in \mathbb{N} $.
The operators 
$ P_N $, $ N \in \mathbb{N} $,
are thus compact linear operators
and from the Daugavet property
of the $ \mathbb{R} $-Banach
space
$
  C( [0,1]^d, \mathbb{R} )
$
(see, e.g., Definition~2.1
and Example (a)
in Werner~\cite{Werner2001})
we get that
$
  \|
    I - P_N
  \|_{ 
    L( C( [0,1]^d, \mathbb{R} ) ) 
  }
  =
  1 +
  \|
    P_N
  \|_{ 
    L( C( [0,1]^d, \mathbb{R} ) ) 
  }
$
for all
$ N \in \mathbb{N} $
(see, for instance, 
Theorem~2.7 in
Werner~\cite{Werner2001}).
Next let 
$ 
  S \colon (0,T]
  \rightarrow L( 
    C( [0,1]^d , \mathbb{R} ) 
  ) 
$ 
be a mapping given by
\begin{equation}
\label{eq:defS}
  (S_t v)(x)
  =
  \sum_{i \in \mathbb{N}^d}
  e^{ - \lambda_i t }
  \int_{(0,1)^d} e_i(s) \, v(s) \, ds
  \cdot e_i(x)
\end{equation}
for all $ t \in (0,T] $,
$ x \in [0,1]^d $
and all 
$ 
  v \in 
  C( [0,1]^d , \mathbb{R} ) 
$.

\begin{lemma}
\label{constructS}
Let $ d \in \left\{ 1,2,3 \right\} $.
Then the mapping
$ 
  S \colon (0,T]
  \rightarrow L( 
    C( [0,1]^d , \mathbb{R} ) 
  ) 
$ 
given by \eqref{eq:defS}
satisfies
Assumption~\ref{semigroup}
for every 
$
  \alpha \in 
  [
    \frac{ d }{ 4 } + \frac{ \gamma }{ 2 },
    1 
  )
$
and every
$ 
  \gamma \in ( 0, 2 - \frac{ d }{ 2 } )
$.
\end{lemma}
Clearly, this is simply
the semigroup generated by
the Laplacian with Dirichlet
boundary conditions
(see, e.g., Section 
3.8.1 in \cite{sy02}).
Other boundary
conditions such as Neumann
or periodic boundary
conditions could also be 
considered here.
The proof of 
Lemma~\ref{constructS}
is well-known
and therefore
omitted.
We also add that the 
proof of 
Lemma~\ref{constructS}
essentially uses a suitable
Sobolev embedding
and for this the
condition $ d \leq 3 $
is assumed
in Lemma~\ref{constructS}.
We now present the promised
example of 
Assumption~\ref{stochconv}.
We consider a stochastic convolution 
of the semigroup $ S $ constructed
in \eqref{eq:defS}
and a cylindrical Wiener process.
The following result provides
an appropriate version of such a
process, in which 
the initial value
of the stochastic evolution
equation \eqref{eqspdesol} 
is additionally incorporated.

\begin{proposition}
\label{constructO}
Let $ d \in \mathbb{N} $,
let 
$ 
  V = C( [0,1]^d, \mathbb{R}
  ) 
$
with
$ 
  \| v \|_V =
  \| v \|_{ C( [0,1]^d, \mathbb{R} ) }
$ 
for every
$ v \in V $, 
let 
$ 
  \rho \in (0,\infty)
$,
let 
$ 
  \beta^i \colon [0,T] 
  \times \Omega \rightarrow
  \mathbb{R} 
$, 
$ 
  i \in \mathbb{N}^d
$,
be a family 
of independent 
standard Brownian motions
with continuous sample paths
and 
let 
$ 
  b \colon \mathbb{N}^d
  \rightarrow \mathbb{R} 
$
be a function with
$
  \sum_{ 
    i \in \mathbb{N}^d
  }
  \left( i_1^2 + \ldots + i_d^2 \right)^{ 
    \left( \rho - 1 \right) }
  | b(i) |^2 
  < \infty 
$.
Furthermore, suppose that 
$ 
  \xi \colon \Omega
  \rightarrow V 
$ 
is an
$ 
  \mathcal{F}
$/$
  \mathcal{B}(V)
$-measurable
mapping with 
$
  \sup_{N \in \mathbb{N}}
  \left(
    N^{ \rho }
    \left\| \xi(\omega) - 
      P_N( \xi(\omega)
      ) 
    \right\|_V
  \right) < \infty 
$
for every 
$ \omega \in \Omega $.
Then there exists 
an up to indistinguishability
unique stochastic 
process 
$ 
  O \colon [0,T]
  \times \Omega \rightarrow V
$
with continuous sample paths
which satisfies
\begin{equation}\label{approxN}
  \mathbb{P}\bigg[
    \lim_{N \rightarrow \infty}\,
      \sup_{0 < t \leq T}
      \Big\| O_t - S_t \, \xi 
        - 
        \!\!\!\!\!\!
        \sum_{ 
          i \in 
          \{ 1, \dots, N \}^d 
        }
        \!\!\!\!\!\!
        b(i) \,
        \Big( -\lambda_i 
        \int^t_0
        e^{ 
        - \lambda_i (t-s) }
        \beta^i_s \,ds
        + \beta^i_t 
        \Big) \,
        e_i
      \Big\|_{V}
    = 0
  \bigg] = 1
\end{equation}
and
\begin{equation}
  \sup_{ N \in \mathbb{N} }
  \sup_{ 0 \leq t \leq T }
  \big(
    N^{\gamma}
    \|
      O_t(\omega) - 
      P_N( O_t(\omega)
      )
    \|_V 
  \big)
  < \infty
\end{equation}
for every $ \omega \in \Omega $ and
every $\gamma \in (0,\rho)$.
In particular, $O$ satisfies
Assumption \ref{stochconv}
for every $ \gamma \in (0,\rho) $.
Here the functions 
$ e_i \in V $,
$ i \in \mathbb{N}^d $,
the real numbers
$ \lambda_i $, $ i \in \mathbb{N}^d$,
and the linear operators 
$ 
  P_N \colon V \rightarrow V 
$,
$ 
  N \in \mathbb{N}
$, 
are given
in 
\eqref{deflambdai}
and \eqref{PNex}.
\end{proposition}

Proposition~\ref{constructO}
follows directly from 
Lemma~\ref{constructOO} below.
Let us add some remarks
concerning 
Proposition~\ref{constructO}.
Let
$ 
  L^2( (0,1)^d , \mathbb{R} ) 
$ 
be the 
$ \mathbb{R} $-Hilbert 
space of equivalence classes 
of
$ 
  \mathcal{B}( (0,1)^d ) 
$/$
  \mathcal{B}( \mathbb{R} )
$-measurable 
and Lebesgue square integral 
functions from $ (0,1)^d $
to 
$ \mathbb{R} $
and let
$ 
  B \colon 
  L^2( (0,1)^d, \mathbb{R})
  \rightarrow
  L^2( (0,1)^d, \mathbb{R})
$ 
be a bounded linear
operator given by
\begin{equation}
\label{Bdef}
  B v = \sum_{ i \in \mathbb{N}^d }
  b(i) 
    \int_{(0,1)^d} e_i(s) \, v(s) \, ds
  \cdot
  e_i
\end{equation}
for all 
$ 
  v \in 
  L^2( (0,1)^d, \mathbb{R} ) 
$
where
$ 
  b \colon 
  \mathbb{N}^d \rightarrow \mathbb{R} 
$
is the function used in 
Proposition~\ref{constructO}.
Then the stochastic
process
$ 
  O \colon [0,T] \times \Omega
  \rightarrow 
  C( [0,1]^d, \mathbb{R} 
  )
$
in Proposition~\ref{constructO}
satisfies
\begin{equation}
\begin{split}
  O_t 
& =
  S_t \, \xi + 
  \sum_{ i \in \mathbb{N}^d }
  b(i) 
    \int^t_0 e^{
      -\lambda_i(t-s) } \,
    d\beta^i_s 
\cdot
  e_i 
  =
  S_t \, \xi + 
  \int_0^t 
  S_{ t - s } \, B \, dW_s
\end{split}
\end{equation}
$ \mathbb{P} $-a.s.\ for
every $ t \in (0,T] $
where 
$ W_t $, $ t \in [0,T]$,
is an appropriate 
cylindrical $I$-Wiener process on 
$ 
  L^2( (0,1)^d, \mathbb{R}
  ) 
$.
In particular,
$ 
  O \colon [0,T] \times \Omega
  \rightarrow 
  C( [0,1]^d, \mathbb{R} 
  )
$
is the up to indistinguishability
unique mild solution process 
of the 
linear SPDE
\begin{equation}
\label{eq:linearSPDE}
  d O_t = \big[ \Delta O_t \big] \, dt
  + B \, dW_t,
\qquad
  O_t|_{ \partial (0,1)^d } \equiv 0,
\qquad
  O_0 = \xi
\end{equation}
for 
$ 
  t \in [0,T] 
$
on
$ 
  C( [0,1]^d, \mathbb{R} ) 
$.
The process $O$ thus includes
the initial value and a
stochastic convolution
of the semigroup generated
by the Laplacian with
Dirichlet boundary conditions 
and a cylindrical Wiener process
as it is frequently considered
in the literature 
(see, e.g., 
Section 5 in \cite{dz92}).
Note also that the
operator $ B $ 
appearing
in 
\eqref{eq:linearSPDE}
is diagonal
with respect to
the orthonormal
basis
$
  e_i \in
  L^2( (0,1)^d, \mathbb{R} )
$,
$ 
  i \in \mathbb{N}^d
$,
in
$
  L^2( (0,1)^d, \mathbb{R} )
$.
This assumption is
strongly exploited in
Proposition~\ref{constructO}.
However, the abstract
setting in Section~\ref{sec2}
does not need this
assumption to be fulfilled
and, in principle, linear
operators $ B $ that
are not diagonal with
respect to 
$ e_i $, $ i \in \mathbb{N}^d $,
could be considered here.
The detailed analysis in
the non-diagonal case
remains an open question
for future research. 
The reader is referred to \cite{DBKam} 
for first results in that direction.


To illustrate 
Proposition~\ref{constructO}
we consider the following 
simple example.
If $ d = 2 $, 
$ 
  ( \xi(\omega) )(x) = 0 
$
for all 
$ x \in [0,1]^2 $, 
$ 
  \omega
  \in \Omega 
$ 
and 
$ 
  b( (i_1,i_2) )
  = \frac{1}{ ( i_1 + i_2 ) } 
$
for all 
$ 
  i = (i_1, i_2) 
  \in \mathbb{N}^2 
$
in Proposition~\ref{constructO},
then Proposition~\ref{constructO} 
implies the existence of
$ 
  \mathcal{F}
$/$
  \mathcal{B}([0,\infty))
$-measurable
mappings 
$ 
  C_\gamma 
  \colon 
  \Omega \rightarrow
  [0,\infty) 
$,
$ 
  \gamma \in (0,1)
$,
such that
\begin{equation}
\label{eq:Oerrorrate}
    \sup_{ 0 \leq t \leq 1 }
    \sup_{ x \in [0,1]^2 }
    \big| 
      O_t(\omega,x) - 
      ( P_N O_t )(\omega,x)
    \big|
  \leq
  C_{\gamma}(\omega)  
  \cdot
  N^{-\gamma}
\end{equation}
for all $ \omega \in \Omega $,
$ N \in \mathbb{N} $
and all 
$ \gamma \in (0,1) $.
Finally, note that 
Proposition~\ref{constructO} 
follows immediately
from the next result (Lemma~\ref{constructOO}), 
which is also of independent interest.
Its proof is postponed to Subsection~\ref{sec:constructOO}.
Estimates related to Lemma~\ref{constructOO} 
and its proof can, e.g., be found in
Section 5.5.1 in 
Da Prato \& 
Zabczyk~\cite{dz92} 
and in Proposition~1.1 
and Proposition~1.2 
in Da Prato
\& Debussche~\cite{GdP-AD:96}. 
In particular, the temporal regularity 
statements in Lemma~\ref{constructOO}
follow, e.g., immediately from 
Lemma~5.19 and Theorem~5.20
in \cite{dz92}.

\begin{lemma}
\label{constructOO}
Let $ d \in \mathbb{N}$,
let 
$ 
  V = C( [0,1]^d, \mathbb{R}
  ) 
$
with
$ 
  \| v \|_V =
  \| v \|_{ C( [0,1]^d, \mathbb{R} ) }
$ 
for every
$ v \in V $, 
let $ \rho \in (0,\infty) $,
let 
$ 
  \beta^i \colon [0,T] 
  \times \Omega \rightarrow
  \mathbb{R} 
$, 
$ 
  i \in \mathbb{N}^d
$,
be a family 
of independent 
standard Brownian motions
with continuous sample paths
and 
let 
$ 
  b \colon \mathbb{N}^d
  \rightarrow \mathbb{R} 
$
be a function with
$
  \sum_{ 
    i \in \mathbb{N}^d
  }
  \left( i_1^2 + \ldots + i_d^2 \right)^{ 
    \left( \rho - 1 \right) }
  | b(i) |^2 
  < \infty 
$.
Then
there exists an
up to indistinguishability
unique stochastic
process
$
  O \colon [0,T] \times \Omega
  \rightarrow V
$
which satisfies
\begin{equation*}
  \sup_{ N \in \mathbb{N} }
  \sup_{ 0 \leq t \leq T }\!\big(
  N^{\gamma} 
  \left\|
    O_t(\omega) - P_N( O_t(\omega) )
  \right\|_V 
  \!
  \big)
  +
  \sup_{0 \leq t_1 < t_2 \leq T}
  \frac{
    \left\| O_{t_2}(\omega)
    - O_{t_1}(\omega) \right\|_V
  }{
    \left| t_2 - t_1 \right|^\theta
  } 
  < \infty
\end{equation*}
for every 
$ \omega \in \Omega $,
every
$ 
  \theta \in 
  (
    0,
    \min(\frac{1}{2},\frac{\rho}{2}
    )
  )
$, 
every 
$
  \gamma \in (0,\rho)
$
and which satisfies
\begin{equation*}
  \sup_{ N \in \mathbb{N} }
  \Big\{
  N^{\gamma}
 \Big(
  \mathbb{E}\Big[
    \sup_{0\leq t \leq T}
    \| O_t - P_N( O_t )\|_V^p\Big]
  \Big)^{ \! \frac{1}{p} }
  \Big\}
  +
  \sup_{0 \leq t_1 < t_2 \leq T}\!\!
  \frac{
    \left(
    \mathbb{E}\!\left[
      \| O_{t_2} - O_{t_1} \|_V^p
    \right]
    \right)^{\frac{1}{p}}
  }{
    \left| 
      t_2 - t_1 
    \right|^{\theta}
  }
  < \infty
\end{equation*}
and
\begin{equation*}
  \mathbb{P}\!\left[
    \lim_{N \rightarrow \infty}\,
      \sup_{0\leq t \leq T}
      \Big\| O_t - 
        \sum_{ i \in 
        \left\{ 1, \dots, N 
        \right\}^d }
        b(i) 
        \Big( 
        -\lambda_i
        \int^t_0
        e^{ 
        - \lambda_i (t-s) }
        \beta^i_s \, ds
        + \beta^i_t \Big) 
        \cdot e_i
      \Big\|_V
    = 0
  \right] = 1
\end{equation*}
for every $ p \in [1,\infty)$,
every 
$ 
  \theta \in 
  (0,\frac{\rho}{2})
  \cap
  [0,\frac{1}{2}]
$
and every 
$  
  \gamma \in (0,\rho)
$.
Here 
the functions $ e_i \in V $,
$ i \in \mathbb{N}^d $,
the real numbers
$ \lambda_i $, $ i \in \mathbb{N}^d$,
and the linear operators 
$ 
  P_N \colon V \rightarrow V 
$,
$ N \in \mathbb{N}$, 
are 
given
in 
\eqref{deflambdai}
and \eqref{PNex}.
\end{lemma}

%
%
\subsection{Stochastic
evolution equations with
a globally Lipschitz nonlinearity}
\label{stochreaction}
If the nonlinearity
$ 
  F \colon V \rightarrow W 
$ 
given
in Assumption~\ref{drift}
is globally Lipschitz continuous
from $V$ to $W$, 
then Assumption~\ref{solution}
is naturally met.
\begin{proposition}
\label{Lipschitz}
Suppose that Assumptions
\ref{semigroup}-\ref{stochconv}
are fulfilled.
If the nonlinearity 
$ F \colon V \rightarrow W $
given in Assumption \ref{drift}
additionally satisfies
$
  \sup_{ v, w \in V, v \neq w }
  \frac{\| F(v) - F(w)\|_W}{\| v - w \|_V }
  < \infty ,
$
then Assumption \ref{solution}
is fulfilled.
\end{proposition}

The proof of Proposition~\ref{Lipschitz}
is straightforward and therefore omitted.
In the remainder of this section 
we illustrate Theorem \ref{mainthm}
with a stochastic reaction diffusion
equation with a globally Lipschitz nonlinearity.
The next lemma describes
the nonlinearities considered
in this subsection.
Its proof is clear and
hence omitted.

\begin{lemma}\label{constructF}
Let $ d \in \mathbb{N} $ and
let $ f \colon [0,1]^d 
\times \mathbb{R}
\rightarrow \mathbb{R} $ be a
continuous function which
satisfies
\begin{equation}
  \sup_{ x \in [0,1]^d }
  \sup_{ 
\substack{
      y_1, y_2 \in \mathbb{R}\\
      y_1 \neq y_2
     }}
  \frac{| f(x,y_1) - f(x,y_2)|}{| y_1 - y_2 |}
  < \infty .
\end{equation}
Then the corresponding Nemytskii
operator
$ 
  F \colon C ( [0,1]^d, \mathbb{R}  )
  \rightarrow  C ( [0,1]^d, \mathbb{R}  )
$ 
given by
$
  ( F( v ))(x)= f( x, v(x) )
$
for every
$  x \in [0,1]^d
$
and every 
$ v \in C( [0,1]^d, \mathbb{R}) $
satisfies 
\begin{equation}
  \sup_{ 
    v, w \in V, v \neq w 
  } 
  \frac{
    \| F(v) - F(w) \|_{
      C( [0,1]^d, \mathbb{R} )
    }   
  }{ 
    \| v - w \|_{ 
      C( [0,1]^d, \mathbb{R} ) 
    }
  } 
  < \infty .
\end{equation}
\end{lemma}
Let 
$ d \in \{ 1, 2, 3 \} $
and
$ 
  V = W = C( [0,1]^d, \mathbb{R} ) 
$
and let
$ 
  P_N \colon V \rightarrow V 
$, 
$
  N \in \mathbb{N}
$,
$ 
  S \colon (0,T] \rightarrow L(V) 
$,
$ 
  F \colon V \rightarrow V 
$
and 
$ 
  O \colon [0,T] \times \Omega 
  \rightarrow V 
$
be given by 
\eqref{PNex}, 
by \eqref{eq:defS},
by
Lemma~\ref{constructF}
and 
by Proposition~\ref{constructO}.
Then
Assumption~\ref{solution} 
is fulfilled due to 
Proposition~\ref{Lipschitz}
and therefore, the 
assumptions in
Theorem~\ref{mainthm} 
are fulfilled.
In addition, the stochastic 
evolution 
equation~\eqref{eqspdesol}
reduces in this case to
\begin{equation}
\label{herespde2}
  d X_t
  = 
  \big[
    \Delta
    X_t
  +
    f( \cdot, X_t ) 
  \big] \, dt 
  +
  B\,dW_t,
\qquad
  X_t|_{\partial (0,1)^d} \equiv 0,
\qquad
  X_0 = \xi
\end{equation}
for $ t \in [0,T]$, where
$ W_t $, $ t \in [0,T] $, is a 
cylindrical $I$-Wiener process
on $L^2( (0,1)^d, \mathbb{R}
) $, where
$ 
  \xi \colon \Omega \rightarrow V 
$
is used in 
Proposition~\ref{constructO}
and where the bounded linear operator
$ 
  B \colon 
  L^2( (0,1)^d, \mathbb{R} )
$
$
  \rightarrow 
$
$
  L^2( (0,1)^d, \mathbb{R} 
  )
$
is given by \eqref{Bdef}
with 
$ 
  b \colon \mathbb{N}^d 
  \rightarrow \mathbb{R} 
$
used in 
Proposition~\ref{constructO}.
Moreover, the finite dimensional 
SODEs~\eqref{GalSODEs} 
reduce to
\begin{equation*}
  d X^N_t = 
  \left[ 
    \Delta X^N_t
    + P_N f(\cdot, X^N_t) 
  \right] dt
  +
  P_N B \, dW_t,
\quad
  X^N_t|_{\partial (0,1)^d}
  \equiv
  0 ,
\quad
  X_0^N = P_N (\xi)
\end{equation*}
for $ t \in [0,T]$
and $ N \in \mathbb{N} $.
If $ d = 1 $ and
$ b(i) = b(1) $ 
for all $ i \in \mathbb{N} $,
then Lemma~\ref{constructS},
Proposition~\ref{constructO}
and Theorem~\ref{mainthm} 
yield the existence of 
$ \mathcal{F} $/$\mathcal{B}([0,\infty))$-measurable
mappings 
$ 
  C_\gamma \colon \Omega 
  \rightarrow [0,\infty) 
$,
$ 
  \gamma \in (0,\frac{1}{2})
$,
such that
\begin{equation}\label{resex1}
  \sup_{0 \leq t \leq T}
  \sup_{0 \leq x \leq 1}
  \big|
    X_t(\omega,x) - X^N_t(\omega,x)
  \big|
  \leq
  C_\gamma(\omega)
  \cdot
  N^{ - \gamma }
\end{equation}
for every $ \omega \in \Omega $,
$ N \in \mathbb{N} $
and every $ \gamma \in ( 0, \frac{ 1 }{ 2 } ) $.
Hence, in the case $ d = 1 $ and
$ b(i) = b(1) $ 
for all $ i \in \mathbb{N} $,
we obtain that
$ X^N_t(\omega,x) $
converges to $ X_t(\omega,x) $ uniformly
in $ t \in [0,T] $ and
$ x \in [0,1] $ with the
rate $ \frac{ 1 }{ 2 } - $
as $ N $ goes to infinity
for every $ \omega \in \Omega $.

\subsection{Stochastic
Burgers equation}
\label{stochburger}
In this subsection a stochastic
Burgers equation 
is formulated in the setting of
Section~\ref{sec2}.
For this a few function spaces
from the literature
(see, e.g., Chapter~5 in \cite{r01})
are presented first.
By
$
  ( 
    L^2( (0,1), \mathbb{R} ) ,
$
$
    \left\| \cdot \right\|_{L^2},
$
$
    \left< \cdot, \cdot \right>_{L^2}
  )
$
the $ \mathbb{R} $-Hilbert space of
equivalence classes of 
$
  \mathcal{B}(0,1)
$/$
  \mathcal{B}(\mathbb{R})
$-measurable
and Lebesgue square integrable
functions from $ (0,1) $ to $ \mathbb{R} $
with 
scalar product 
$
  \left< v, w \right>_{ L^2 }
  :=
  \int^1_0 v(x) \, w(x) \, dx
$
and norm
$
  \| v \|_{ L^2 }
  :=(\left< v, v \right>_{ L^2 })^{ 1/2 }
$
for every 
$ 
  v, w \in 
  L^2( (0,1), \mathbb{R} ) 
$
is denoted.
In addition, by
$
  H^1( (0,1), \mathbb{R}
  )
$
the Sobolev space
of weakly differentiable 
functions from $ (0,1) $ to $ \mathbb{R} $ with weak
derivatives in
$
  L^2( (0,1), \mathbb{R} )
$
is denoted.
The norm and the scalar product
in 
$
  H^1( (0,1), \mathbb{R}
  )
$
are defined by
$
  \| v \|_{ H^1 }
  :=
  (
    \| 
      v 
    \|_{ L^2 }^2
    +
    \| 
      v' 
    \|_{ L^2 }^2
  )^{ 
    1 / 2  
  } 
$
and
$
  \left< v, w \right>_{ H^1 }
  :=
  \left< 
    v ,
    w
  \right>_{ L^2 } 
  +
  \left< 
    v' ,
    w'
  \right>_{ L^2 } 
$
for every 
$ 
  v, w \in 
  H^1( (0,1), \mathbb{R} ) 
$.
Additionally, by 
$
  H^1_0( (0,1), \mathbb{R} )
$
the closure of 
$ 
  C^\infty_{\mathrm{cpt}}( 
    (0,1), 
    \mathbb{R}
  )
$
in the 
$\mathbb{R}$-Hilbert space
$
  (
$
$
    H^1( 
      (0,1), \mathbb{R}
    ),
$
$
    \left\| \cdot \right\|_{ H^1 }
    , 
$
$
    \left< \cdot, \cdot \right>_{ H^1 }
  )
$ 
is denoted
and the norm and the scalar
product in
$
  H^1_0( (0,1), \mathbb{R}
  ) 
$
are denoted by
$
  \| v \|_{ H^1_0 }
  :=
  \| v' \|_{L^2}
$
and
$
  \left< v, w \right>_{ H^1_0 }
  :=
  \left< v', w' \right>_{ L^2 }
$
for every 
$
  v, w \in 
  H^1_0( (0,1), \mathbb{R}
  ) 
$. 
The Sobolev space
$
  (
    H^{ - 1 }((0,1), \mathbb{R} ),
    \left\| \cdot \right\|_{ H^{ - 1 } }
  )
$
$
  :=
$
$
  (
    H^1_0( (0,1), \mathbb{R}
    ), \left\| \cdot \right\|_{ H^1_0 }
  )'
$
is also used below and by
$ 
  \partial \colon 
  L^2( 
    (0,1), \mathbb{R}
  )
  \rightarrow
  H^{-1}( ( 0,1), \mathbb{R} )
$ 
the distributional derivative in
$
  L^2( 
    (0,1), \mathbb{R}
  )
$
defined by
$
  ( \partial v  )(\varphi)
  = ( v' )(\varphi)
  := - \left< v, \varphi' \right>_{L^2}
$
for every 
$ 
  \varphi \in 
  H^1_0( (0,1), \mathbb{R}
  ) 
$ 
and every
$ 
  v \in 
  L^2( (0,1), \mathbb{R} ) 
$
is denoted.

In view of these 
function spaces, let 
$ 
  W = 
  H^{-1}( (0,1) , \mathbb{R}
  )
$
with
$
  \| v \|_W
  :=
  \| v \|_{ H^{ - 1 } }
$
for all $ v \in W $
and let
$ 
  V = 
  C( [0,1],
    \mathbb{R} 
  ) 
$
with
$
  \| v \|_V :=
  \sup_{ x \in [0,1] }
  | v(x) |
$
for all $ v \in V $
be the $\mathbb{R}$-Banach 
space of continuous
functions from $ [0,1] $
to $ \mathbb{R} $.
As in Sections~\ref{stochheat} and \ref{stochreaction},
we use
the projection operators
$ 
  P_N \colon 
  C( [0,1], \mathbb{R}
  ) 
  \rightarrow 
  C( [0,1], \mathbb{R})
$,
$ 
  N \in \mathbb{N} 
$,
defined by
\begin{equation}
\label{projoperator}
  \big( P_N( v ) \big)(x)
  :=
  \sum^N_{n=1}
  2
    \int^1_0
      \sin( n \pi s ) \, v(s) \,
    ds
  \cdot
  \sin( n \pi x )
\end{equation}
for every $ x \in [0,1]$,
$ 
  v \in C ( [0,1], \mathbb{R})  
$ 
and every $N \in \mathbb{N} $.
The semigroup is constructed in
the following well-known lemma here.

\begin{lemma}
\label{BurgerSG}
The mapping 
$ 
  S \colon (0,T] \rightarrow 
  L\big( 
    H^{-1}( (0,1), \mathbb{R}
    ) , 
    C( [0,1], \mathbb{R}
    )
  \big) 
$
given by
$
  \big( S_t(w) \big)( x ) =
  \sum^{\infty}_{n=1}
   2 \cdot e^{ - n^2 \pi^2 t }
  \cdot 
  w(
    \sin( n \pi ( \cdot ) ) 
  )
  \cdot
  \sin( n \pi x ) 
$
for every $ x \in [0,1] $,
$ 
  w \in 
  H^{-1}( 
    (0,1), \mathbb{R}
  ) 
$ 
and every $ t \in (0,T] $
is well defined and satisfies
Assumption~\ref{semigroup}
for every $ \gamma \in (0,\frac{1}{2}) $.
\end{lemma}
The proof of 
Lemma~\ref{BurgerSG}
can be found in 
Subsection~\ref{sec:burgerSG}.
The next well-known lemma describes the nonlinearities
for the stochastic Burgers equations
considered in this section.

\begin{lemma}
\label{BurgerNonl}
Let $ c \in \mathbb{R} $ 
be a real number.
Then the mapping 
$
  F \colon C( [0,1], \mathbb{R} )
  \rightarrow
  H^{-1}( (0,1), \mathbb{R} )
$
given by
$
  F(v) = c \cdot \partial 
  \big( v^2 \big)
$
for every 
$ 
  v \in C( [0,1], 
  \mathbb{R} ) 
$
satisfies Assumption~\ref{drift}.
\end{lemma}

\begin{proof}
The estimate
$ \left\| 
  \partial v
\right\|_{ H^{-1} }
\leq \left\|
  v
\right\|_{L^2}
$
for every 
$ 
  v \in 
  L^2( 
    (0,1), \mathbb{R}
  )
$
implies
\begin{equation}
\begin{split}  
&
  \| F(v) - F(w)\|_{ H^{-1} }
 =
  \| 
    c \, \partial ( v^2 ) - c \, \partial ( w^2 )
  \|_{ H^{-1} }
  \leq  
  |c| \cdot 
  \| v^2  - w^2\|_{ L^2 }
\\ & \leq
  |c| \cdot 
  (
    \| v \|_{ 
      C( [0,1], \mathbb{R} ) 
    }
    +
    \| w \|_{ 
      C( [0,1], \mathbb{R} ) 
    }
  )
  \cdot\|  v - w \|_{ C([0,1],\mathbb{R}) }
\end{split}
\end{equation}
for every 
$ 
  v, w \in C( [0,1], \mathbb{R} ) 
$.
This yields
$
  \left\| 
    F(v) - F(w)
  \right\|_{ H^{-1} }
\leq
  2 \, r \, |c|
  \left\|  
    v - w
  \right\|_{ 
    C( [0,1], \mathbb{R} ) 
  }
$
for every 
$ 
  v, w \in 
  C( [0,1], \mathbb{R}) 
$
with 
$ 
  \| v \|_{ C([0,1],\mathbb{R}) }, 
  \| w \|_{ C([0,1],\mathbb{R})
  } \leq r 
$
and every $ r \in (0,\infty) $.
The proof of
Lemma~\ref{BurgerNonl}
is thus completed.
\end{proof}
For this type of nonlinearities,
Assumption~\ref{solution} is fulfilled,
which can be seen in the 
following lemma.
Its proof
is postponed to 
Subsection~\ref{sec:BurgerSol}
below.
A related result with 
can 
be found in Da Prato,
Debussche \&
Temam~\cite{GdP-AD-RT:94}
(see Lemma~3.1 and Theorem~3.1
in \cite{GdP-AD-RT:94}).

\begin{lemma}
\label{BurgerSol}
Let $ V = C( [0,1], \mathbb{R}) $
with
$ 
  \| v \|_V 
  = \sup_{ 0 \leq x \leq 1 }
  | v(x) | 
$
for all
$ v \in V $,
let
$ 
  W = H^{-1}( (0,1), \mathbb{R} ) 
$
with
$
  \| v \|_W = \| v \|_{ H^{ - 1 } }
$
for all 
$ v \in W $
and let 
$ 
  S \colon (0,T] \rightarrow 
  L( W, V ) 
$,
$ 
  F \colon V \rightarrow W 
$
and
$ 
  P_N \colon V \rightarrow V 
$, 
$
  N \in \mathbb{N}
$,
be given by 
Lemma~\ref{BurgerSG},
Lemma~\ref{BurgerNonl}
and \eqref{projoperator}.
Moreover, 
let
$ 
  O \colon [0,T] \times \Omega 
  \rightarrow V 
$
be an arbitrary 
stochastic process 
with continuous sample
paths and with
$
  \sup_{ N \in \mathbb{N} }
  \sup_{0 \leq t \leq T}
  \left\| 
  P_N( O_t(\omega) ) \right\|_V
  < \infty
$ 
for every
$ \omega \in \Omega $.
Then Assumption \ref{solution} 
is fulfilled.
\end{lemma}

We emphasize that 
Lemma~\ref{BurgerSol}
does not assume that
the driving noise process
$
  O \colon [0,T] \times \Omega
  \rightarrow V
$
is a stochastic convolution
involving a cylindrical Wiener
process as considered
in Proposition~\ref{constructO}.
In particular, Lemma~\ref{BurgerSol}
covers stochastic Burgers equations
driven by fractional Brownian motions.
In the next step the consequences of 
Lemmas~\ref{BurgerSG}--\ref{BurgerSol}
and Theorem~\ref{mainthm}
are illustrated by a
numerical example.

%
%
\subsubsection*{Numerical Example}
%
%
%
%
We consider
the stochastic evolution 
equation \eqref{eqspdesol} with
$ 
  S \colon (0,T] \rightarrow L(W,V) 
$,
$ 
  F \colon V \rightarrow W 
$
and 
$ 
  O \colon [0,T] \times \Omega 
  \rightarrow V 
$ 
given by 
Lemma~\ref{BurgerSG},
Lemma~\ref{BurgerNonl} 
and Proposition~\ref{constructO}
with the parameters
$ c = - 30 $, 
$ T = \frac{1}{20} $,
$ \xi(\omega) = \frac{6}{5} \sin(\pi x) $
for every $ \omega \in \Omega $
and
$ b(i) = \frac{ 1 }{ 3 } $
for every $ i \in \mathbb{N} $.
The stochastic evolution equation 
\eqref{eqspdesol} 
then reduces to
\begin{equation}
\label{stochburgereq}
  d X_t(x) =
  \Big[
    \Delta
    X_t 
    - 60 \cdot 
    X_t \cdot
    X_t' \,
  \Big] dt
  +
  \frac{ 1 }{ 3 } \, dW_t,
\qquad
  X_0(\cdot) = 
  \tfrac{6}{5} \sin( \pi \cdot )
\end{equation}
with
$ X_t(0) = X_t(1) = 0 $
for $ t \in [0,\frac{1}{20}]$
on
$
  C( [0,1], \mathbb{R} )
$
and
the finite dimensional 
SODEs~\eqref{GalSODEs} 
simplify to
\begin{equation}\label{stochburgereqSODE}
  d X^N_t =
  \left[
    \Delta
    X^N_t - 
    60 \cdot 
    P_N\big( 
      X^N_t \cdot
      ( X^N_t )' 
    \big)
  \right] dt
  +
  \tfrac{ 1 }{ 3 } \,
  P_N \, dW_t,
\qquad
  X^N_0( \cdot ) = 
  \tfrac{6}{5} \sin( \pi \cdot )
\end{equation}
with
$
  X^N_t(0) = X^N_t(1) = 0
$
for 
$ t \in [0,\frac{1}{20}] $
and $ N \in \mathbb{N} $
on
$
  C( [0,1], \mathbb{R} )
$.
Here 
$ W_t $, $ t \in [0,\frac{1}{20}] $, 
is a
cylindrical $I$-Wiener process
on 
$ 
  L^2( (0,1), \mathbb{R} ) 
$.
Combining 
Proposition~\ref{constructO}
and
Lemmas~\ref{BurgerSG}--\ref{BurgerSol}
with
Theorem~\ref{mainthm} 
then yields the
existence of an 	
unique solution
process
$ 
  X \colon [0,\frac{1}{20}] \times 
  \Omega \rightarrow
  C( [0,1], \mathbb{R}) 
$
with continuous sample paths
of the SPDE \eqref{stochburgereq}.
Moreover,
Proposition~\ref{constructO},
Lemmas~\ref{BurgerSG}--\ref{BurgerSol}
and Theorem~\ref{mainthm} 
imply the existence of 
$ 
  \mathcal{F} 
$/$
  \mathcal{B}( [0,\infty) )
$-measurable
mappings
$
  C_{ \gamma } \colon
  \Omega \rightarrow [0,\infty)
$,
$
  \gamma \in (0,\frac{1}{2}) 
$,
such that
\begin{equation}
\label{burgerass}
  \sup_{ 0 \leq t \leq \frac{1}{20} }
  \sup_{ 0 \leq x \leq 1 }
  \big| 
    X_t(\omega,x)
    -
    X^N_t(\omega,x)
  \big|
  \leq
  C_\gamma(\omega)
  \cdot
  N^{ - \gamma }
\end{equation}
for every $ N \in \mathbb{N} $,
$ \omega \in \Omega $
and every 
$ \gamma \in (0,\frac{1}{2}) $.
Hence, the solutions 
$ X^N_t(\omega,x) $ of 
the finite dimensional SODEs \eqref{stochburgereqSODE}
converge 
to the solution 
$ 
  X_t(\omega,x)
$ 
of the stochastic Burgers
equation \eqref{stochburgereq}
with the rate $ \frac{1}{2}- $
uniformly in 
$ t \in [0,\frac{1}{20}] $ 
and 
$ x \in [0,1] $
as $ N $ goes to infinity
for every $ \omega \in \Omega $.
In Figure~\ref{fig5}
the pathwise approximation
error
\begin{equation}
\label{eq:burgererror}
  \sup_{ 0 \leq t \leq \frac{1}{20} }
  \sup_{ 0 \leq x \leq 1 }
  \big| 
    X_t(\omega,x)
    -
    X^N_t(\omega,x)
  \big|
\end{equation}
is calculated approximatively
and plotted against
$ 
  N \in \{ 16, 32, 64, \dots, 1024, 2048 \} 
$
and two
random $ \omega \in \Omega $.
More precisely,
in the simulations 
presented in Figure~\ref{fig5},
the quantities~\eqref{eq:burgererror} 
are approximated through the
quantities
\begin{equation}
\label{eq:burgererror2}
  \sup_{ 
    t \in 
    \big\{ 
      \frac{ m }{ 4000 } \colon m \in\{ 0,1, \ldots,200 \} 
    \big\}
  }
  \sup_{ 
    x \in 
    \big\{
      \frac{ k }{ 16385 }
      \colon
      k \in \{0, 1, \ldots,16385 \} 
    \big\} 
  }\!\!\!
  \big| 
    Y^{ 16384, 200 }_t(\omega,x) - 
    Y^{ N, 200 }_t(\omega,x)
  \big|
\end{equation}
for
$ 
  N \in \{ 16, 32, 64, \dots, 1024, 2048 \} 
$
and two random
$ \omega \in \Omega $
where
$ 
  Y^{ N, 200 }_t 
  \colon \Omega 
  \rightarrow 
  P_N
    C( [0,1]^d, \mathbb{R} )
$
with
$  
  Y^{ N, 200 }_t
  \approx
  X^N_t
$
($ N $ Fourier nodes for the 
spatial discretization
and $ 200 $ time steps on the
interval $ [0, \frac{1}{ 20 } ] $
for the temporal discretization)
for
$ 
  N \in \{ 16, 32, 64, \dots, 1024, 2048 \} 
  \cup
  \{ 16384 \}
$
and
$
  t \in 
  \{ 
    0, \frac{ 1 }{ 4000 }, 
    \frac{ 2 }{ 4000 },
    \dots, \frac{ 199 }{ 4000 }, 
    \frac{ 1 }{ 20 } 
  \}
$
are suitable
accelerated
exponential Euler
approximations
(see Section~3 in \cite{jk09b})
for the 
SPDE~\eqref{stochburgereq}.
Figure~\ref{fig5} indicates
that the quantity~\eqref{eq:burgererror}
converges to zero with the
(from \eqref{burgerass})
theoretically
predicated order
$ \frac{1}{2} - $.
%
 \begin{figure}[htp] 
    \includegraphics[width=7.4cm]
    {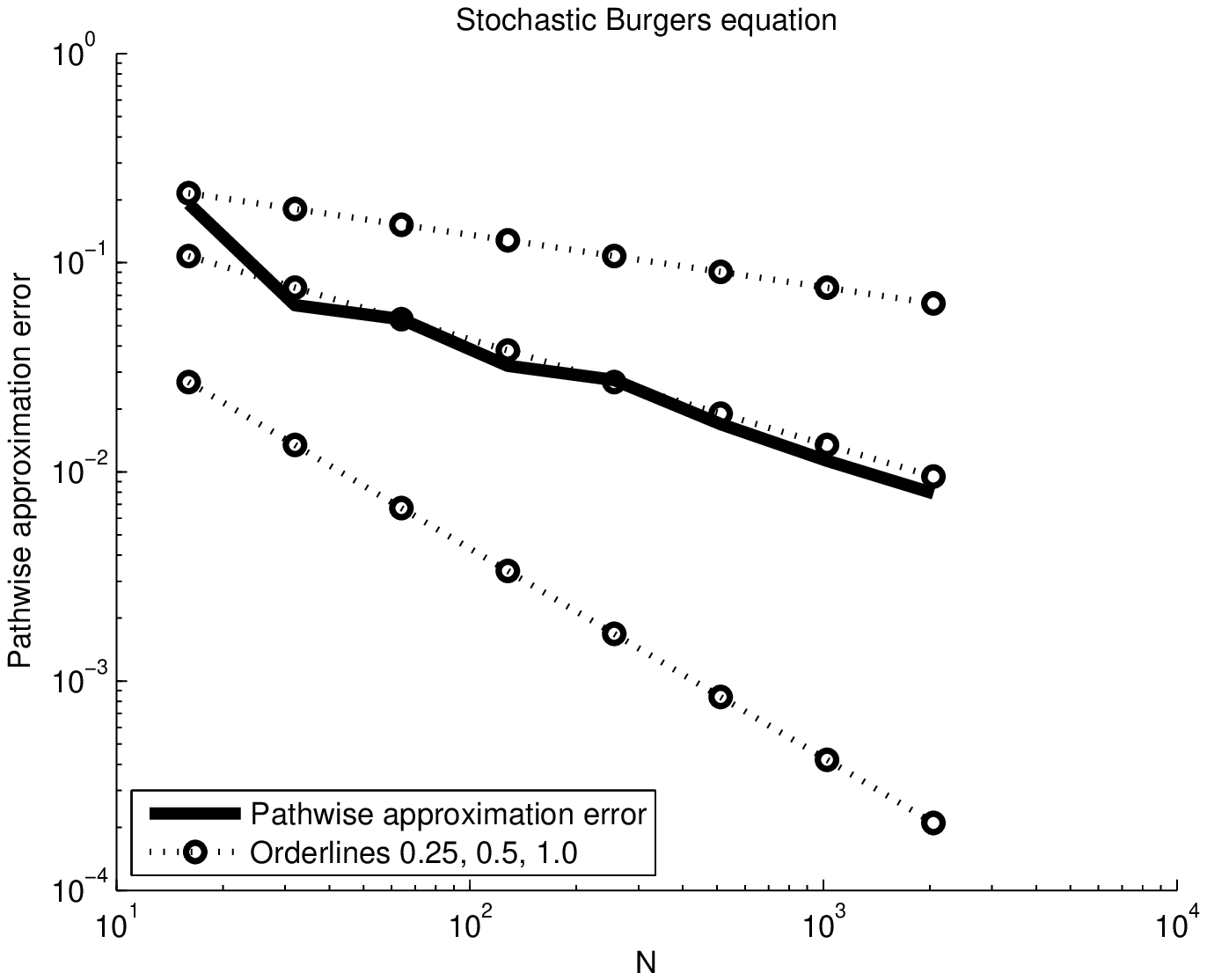}
\hspace*{-1.4cm}
    \includegraphics[width=7.4cm]
    {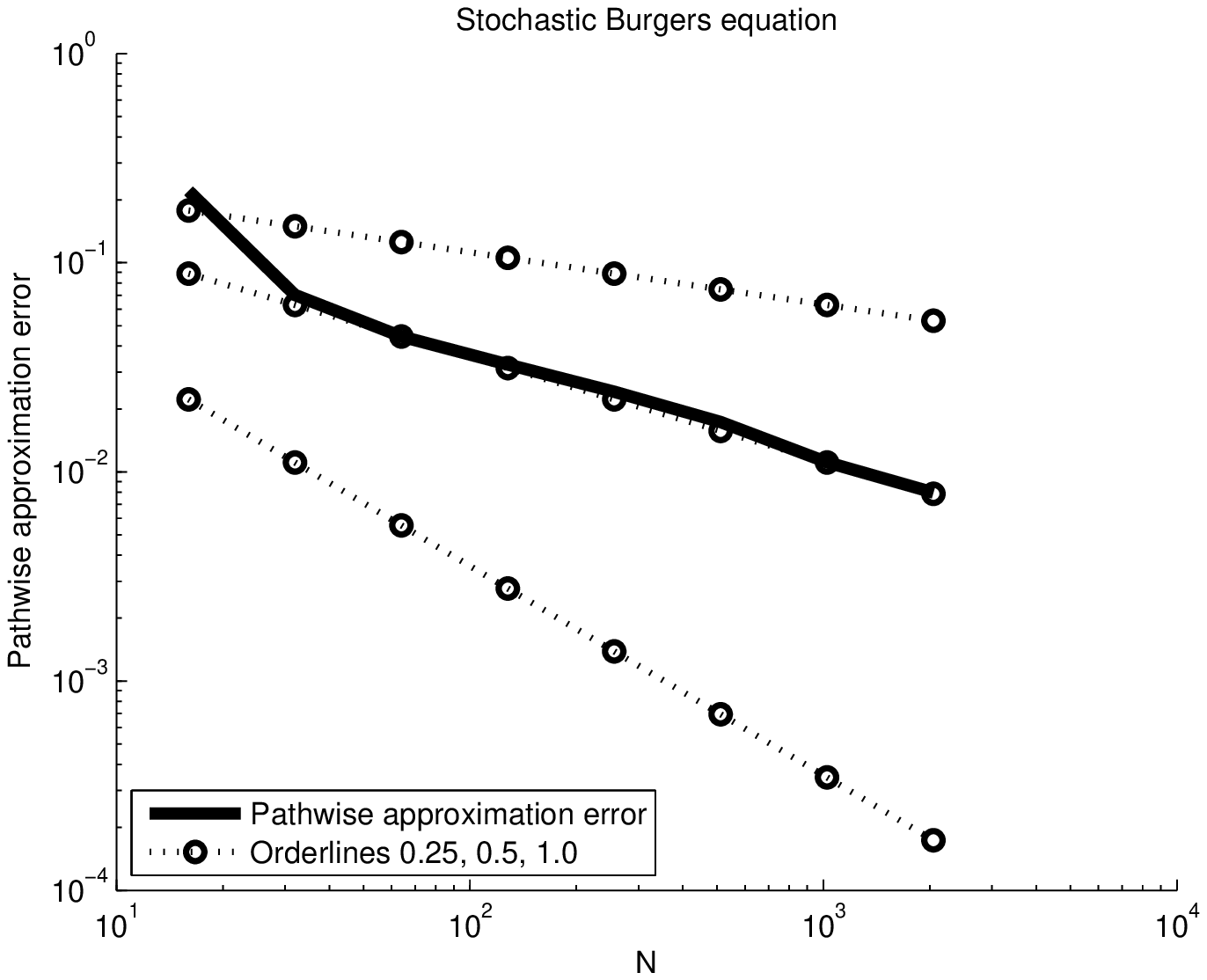}
   \caption{ \label{fig5}
   Pathwise approximation
   error \eqref{eq:burgererror2}
   against $N$ for 
   $ N \in \{ 16,32,64,\dots,1024,2048 \} $
   and two random $\omega \in \Omega $.
   }
 \end{figure}
%

%

\section{Proofs}

In this section we collect all technical proofs of the previous sections.


\subsection{Proof
of Theorem~\ref{mainthm}}
\label{sec:proofthm1}

\begin{proof}
Consider the 
$
  \mathcal{F}
  $/$\mathcal{B}([0,\infty))
$-measurable
mapping 
$ 
  R \colon 
  \Omega\rightarrow [0,\infty)
$ 
defined through
\begin{equation}
\label{eq:defR}
\begin{split}
 & R(\omega) 
  := 
  \sup_{ N \in \mathbb{N} }
  \sup_{ 0 \leq t \leq T }
  \left\| F( X^N_t(\omega) )
  \right\|_W
  +
  T
  +
  \sup_{ N \in \mathbb{N} }
  \sup_{ 0 \leq t \leq T }
  \left(
  N^{ \gamma }
    \left\| 
      O_t(\omega) - 
      P_N( O_t(\omega) 
      )
    \right\|_V
  \right)
\\ & 
  +
  \frac{ 1 }{ (1 - \alpha) }
  +
  \sup_{ N \in \mathbb{N} }
  \sup_{ 0 < t \leq T }
  \left(
  t^{ \alpha }
    \left\| 
      P_N S_t
    \right\|_{L(W,V)}
  \right)
  +
  \sup_{N \in \mathbb{N}}
  \sup_{0 < t \leq T}
  \left(
    t^{\alpha}
    N^{\gamma}
    \left\| S_t - P_N S_t 
    \right\|_{L(W,V)}
  \right)
\end{split}
\end{equation}
for every 
$ \omega \in \Omega $. 
Due
to Assumptions 
\ref{semigroup}-\ref{solution},
the mapping $ R $ 
is indeed finite.
Moreover, note that $ R $ is 
indeed
$\mathcal{F}$/$\mathcal{B}([0,\infty))$-measurable
although $V$ is not assumed
to be separable.
Next consider
the $ \mathcal{B}([0,\infty))
$/$ \mathcal{B}([0,\infty))$-measurable
mapping 
$ 
  L \colon [0,\infty)
  \rightarrow [0,\infty) 
$ 
given by
$
  L(r) :=
  \sup\!\left\{ 
  \frac{ \|F(v) - F(w)\|_W
  }{
    \| v - w \|_V   
  } 
  \colon
  \|v\|_V\leq r, \
  \| w \|_V \leq r, \
      v \neq w
  \right\}
$
for every $ r \in [0,\infty) $.
Additionally,
consider the $\mathcal{F}$/$
\mathcal{B}([0,\infty))$-measurable
mapping $ Z \colon \Omega
\rightarrow [0,\infty)$
given by
$
  Z(\omega)
  :=
  L\!\left(
    \sup_{N \in \mathbb{N}}
    \sup_{0 \leq t \leq T}
    \| X^N_t(\omega)\|_V
  \right)
$
for every $ \omega \in \Omega $.
In the next step
the definition of $ R $
implies
\begin{equation}
\begin{split}
 \|X^N_t  - X^M_t\|_V
& \leq
  \Big\|
    \int^t_0 
      P_N \,
      S_{ t - s } \left( F( X^N_s ) -F( X^M_s ) 
      \right) ds
  \Big\|_V
\\
 & +
  \Big\|
    \int^t_0 ( P_N - P_M )
      S_{ t - s } \,
      F( X^M_s ) \, 
    ds
  \Big\|_V
  +
  R
  \left(
    N^{ - \gamma } + M^{ - \gamma } 
  \right)
\end{split}
\end{equation}
for every $ N, M \in \mathbb{N}$
and every $ t \in [0,T] $
and the estimates
$
  \| 
    ( P_N - I ) \, S_{ t - s }
  \|_{ L(W,V) } 
  \leq
  R 
  N^{-\gamma}( t - s )^{ - \alpha }
$
and
$
\| P_N \,S_{ t - s }\|_{ L(W, V) }
  \leq
  R 
  ( t - s )^{ - \alpha }
$
for every $ N, M \in \mathbb{N} $,
$ s \in [0,t) $
and every $ t \in (0,T] $
therefore show
\begin{equation}
\begin{split}
\| X^N_t  -  X^M_t\|_V
 & \leq
  R
  \int^t_0 
  \tfrac{\|     F( X^N_s ) -  F( X^M_s ) \|_W 
  }{( t - s )^{ \alpha }
  }
  \,
  ds
\\&
  +
  R
  \left(
    N^{ - \gamma } + M^{ - \gamma } 
  \right)
  \int^t_0
    \tfrac{ 
      \left\| F( X^M_s ) \right\|_W 
    }{
      \left( t - s \right)^{ \alpha }
    }
  \, ds
  +
  R
  \left(
    N^{ - \gamma } + M^{ - \gamma } 
  \right)
\end{split}
\end{equation}
for every $ N, M \in \mathbb{N}$
and every $ t \in [0,T] $.
Hence, we have
\begin{equation}
\label{eq:est_mainthm}
  \|X^N_t  - X^M_t\|_V
 	\leq
  R Z
  \int^t_0 
    \| X^N_s - X^M_s 
    \|_V
  \,
  (  t - s )^{ -\alpha }
  \,
  ds
  +
  ( R + R^4 )
( N^{ - \gamma } + M^{ - \gamma } )
\end{equation}
for every $ N, M \in \mathbb{N}$
and every $ t \in [0,T] $
where we used the
estimate
$
  \frac{ 
    T^{ (1 - \alpha) } 
  }{ 
    ( 1 - \alpha ) 
  }
  \leq
  R 
  T^{ (1 - \alpha) }
  \leq
  R^{
    (2 - \alpha) 
  }
  \leq
  R^2
$
in the last inequality
of \eqref{eq:est_mainthm}.  
Lemma~7.1.11 in Henry~\cite{h81} 
hence yields
\begin{align}\label{cauchyS}
  \left\|
    X^N_t  - 
    X^M_t
  \right\|_V
 &\leq 
  \mathrm{E}_{(1 - \alpha)}
  \left(
    t \left( R\, Z\, \Gamma(1 - \alpha) 
    \right)^{\frac{1}{(1 - \alpha)}}
  \right)
  \left( R + R^4 \right)
  \left( N^{-\gamma} + M^{-\gamma} \right)
\nonumber
\\&\leq
  \mathrm{E}_{(1 - \alpha)}
  \left(
     T \left( R\, Z\, \Gamma(1 - \alpha) 
     \right)^{\frac{1}{(1 - \alpha)}}
  \right)
  \left( 2 R^4 \right)
  \left( N^{-\gamma} + M^{-\gamma} \right)
\end{align}
for every $ N, M \in \mathbb{N}$
and every $ t \in [0,T]$.
Here and below the functions 
$
  \mathrm{E}_r \colon [0,\infty)
  \rightarrow [0,\infty)
$,
$ r \in (0,\infty) 
$,
are defined through
$ 
  \mathrm{E}_r( x ) :=
  \sum_{ n = 0 }^{ \infty }
  \frac{
    x^{ n r }
  }{
    \Gamma( n r + 1 )
  }
$
for all 
$ x \in [0,\infty) $
and all
$ r \in (0,\infty) $
(see Lemma~7.1.11 in \cite{h81}
for details).
This shows that 
$ \left( X^N( \omega ) \right)_{ N \in \mathbb{N} } 
$ is a Cauchy-sequence in
$ 
  C( 
    [0,T], V
  ) 
$ 
for every $ \omega \in \Omega $.
Since $ C( [0,T], V ) $
is complete, we can define
the stochastic process
$ X \colon 
[0,T] \times \Omega \rightarrow V $
with continuous sample paths
by
$
  X_t(\omega)
  :=
  \lim_{N \rightarrow \infty}
  X^N_t(\omega)
$
for every $ t \in [0,T]$
and every $\omega \in \Omega$.
Hence, we obtain
\begin{align*}
&
  X_t(\omega)
  =
  \lim\limits_{N\rightarrow \infty}
  X_t^N(\omega)
  =
  \lim\limits_{N\rightarrow \infty}
  \left(
    \int^t_0 
      P_N \, 
      S_{ t - s } \,
      F( X^N_s(\omega) ) \,
    ds
    +
    P_N\!\left( O_t(\omega) \right)
  \right)
\\&=
  \lim\limits_{N\rightarrow \infty}
  \left(
    \int^t_0 
      P_N \, 
      S_{ t - s } \,
      F( X^N_s(\omega) ) \,
    ds
  \right)
    +
    O_t(\omega)
  =
    \int^t_0 
      S_{ t - s } \,
      F( X_s(\omega) ) \,
    ds
    +
    O_t(\omega)
\end{align*}
for every $ t \in [0,T]$
and every $ \omega \in \Omega $.
Moreover, 
if $ Y \colon [0,T] \times
\Omega \rightarrow V $ is a further
stochastic process
with continuous sample paths and with
$
  Y_t( \omega ) = 
    \int^t_0 
      S_{ t - s } \,
      F( Y_s(\omega) ) \,
    ds
    +
    O_t(\omega)
$
for every $ t \in [0,T] $ and
every $ \omega \in \Omega $,
then we obtain
\begin{equation}
\begin{split}
&
\| X_t - Y_t\|_V
  \leq
  R
  \int^t_0 
    ( t - s )^{ - \alpha } 
  \,
    \|  F( X_s ) - F( Y_s ) \|_W  
  \, ds
\\ & \leq 
  R \cdot
  L \Big( 
    \sup_{0 \leq r \leq T }
     \| X_r  \|_V
    +
    \sup_{0 \leq r \leq T }
     \| Y_r  \|_V
  \Big) \cdot
  \int^t_0 
    ( t - s  )^{ - \alpha } \,
    \| X_s -  Y_s \|_V
    \, 
  ds
\end{split}
\end{equation}
for every $ t \in [0,T] $.
Lemma~7.1.11 in \cite{h81}
therefore shows that 
$ X \colon [0,T] \times
\Omega \rightarrow V $ is the 
pathwise unique stochastic process
with continuous sample paths satisfying
equation~\eqref{eqspdesol}.
Moreover, 
\eqref{cauchyS}
yields
$
  \sup_{ 0 \leq t \leq T }
  \|X_t- X^N_t\|_V
	\leq C \cdot N^{-\gamma}
$
for every $N \in \mathbb{N}$,
where the $ \mathcal{F}$/$
\mathcal{B}[0,\infty))$-measurable
mapping 
$ 
  C \colon \Omega
  \rightarrow [0,\infty)
$
is given by
\begin{equation}
\label{eq:described_constant}
  C( \omega )
:=
  2 \cdot 
  \left( 
    R(\omega) 
  \right)^4 
  \cdot
  \mathrm{ E }_{ ( 1 - \alpha ) 
  }\left(
    T 
    \left( 
      R(\omega) 
      Z(\omega) 
      \Gamma( 1 - \alpha )
    \right)^{
      \frac{ 1 }{ 
        \left( 1 - \alpha 
        \right) 
      }
    }
  \right)
\end{equation}
for every $ \omega \in \Omega $.
The proof
of Theorem~\ref{mainthm}
is thus completed.
\end{proof}


\subsection{Proofs
for Subsection~\ref{stochheat}}

\subsubsection{Proof of
Lemma \ref{constructOO}}
\label{sec:constructOO}
Throughout this subsection we use the notation 
$
  \left\| x \right\|_2
  :=
  \left(
    x_1^2 + \ldots + x_d^2
  \right)^{ \frac{1}{2} }
$
for every 
$ x = \left( x_1,
\ldots, x_d \right) \in \mathbb{R}^d $.
We first present three 
elementary lemmas, 
which we need in the 
proof of Lemma \ref{constructOO}.
They are, for example, proved
as Lemmas~9, 11 
and 12 in \cite{j09b}.
%
%
%

\begin{lemma}\label{esti2}
It holds that
$
  \int_{(0,1)^d}
  \int_{(0,1)^d} 
    \frac{ 1 }{
      \left\| x - y \right\|_2^{ \alpha }
    }
  \, dx \, dy
  \leq
  \frac{ \left( 3 d \right)^d 
  }{ \left( d - \alpha \right) }
$
for every $ \alpha \in (0,d) $
and every $ d \in \mathbb{N} $.
\end{lemma}

%
%

\begin{lemma}
\label{lipbasis}
Let 
$ 
  d \in \mathbb{N} 
$ 
and let
$ 
  e_i \in
  C( [0,1]^d, \mathbb{R} )
$, 
$ 
  i \in \mathbb{N}^d 
$,
be given by \eqref{deflambdai}.
Then
$
  \left|
    e_i(x)
    - e_i(y)
  \right|
  \leq
  2^{\frac{d}{2}}
  \pi
  \| i \|_2\| x - y \|_2
$
for every 
$ 
  x, y \in [0,1]^d 
$
and every 
$ 
  i
  \in \mathbb{N}^d 
$.
\end{lemma}


%
%
\begin{lemma}
\label{oholder}
Let 
$ 
  \beta \colon [0,T]
  \times \Omega \rightarrow
  \mathbb{R} 
$ 
be a
standard Brownian motion.
Then
$
  \mathbb{E}\!\left[ 
    |
    \int^{t_2}_0
    e^{ - \lambda (t_2-s) } 
    \, d\beta_s
    -
    \int^{t_1}_0
    e^{ - \lambda (t_1-s) } 
    \, d\beta_s
    |^2
  \right]
  \leq
  \lambda^{ \left( r - 1 \right) } 
  \left| t_2 - t_1 \right|^r
$
for every 
$ t_1, t_2 \in [0,T] $,
$ r \in [0,1] $
and every $ \lambda \in (0,\infty) $.
\end{lemma}
%
%

%
After these three very simple lemmas,
we present now two lemmas (Lemma~\ref{oholder2} 
and Lemma~\ref{oGalerkin}),
which are the essential constituents
in the proof of 
Lemma~\ref{constructOO}.
The first one will ensure 
the temporal
regularity of the 
processes that are
constructed
in Lemma~\ref{constructOO}.

\begin{lemma}
\label{oholder2}
Let $ d \in \mathbb{N}$,
let 
$ 
  \beta^i \colon [0,T] 
  \times \Omega \rightarrow
  \mathbb{R} 
$, 
$ 
  i \in \mathbb{N}^d
$,
be a family 
of independent 
standard Brownian motions
and 
let 
$ 
  b \colon \mathbb{N}^d
  \rightarrow \mathbb{R} 
$
be an arbitrary function.
Then
\begin{equation}
\label{eq:toshow1}
  \Big(
  \mathbb{E}\Big[
    \sup_{ x \in [0,1]^d } 
    \left| O^N_{t_2}(x) 
    - O^N_{t_1}(x) \right|^p
  \Big]
  \Big)^{ \! \frac{1}{p} }
  \leq 
  C_\star
  \left[
    \sum_{ i \in \{1,\dots,N\}^d}
      | b(i) |^2 \,
      \| i \|_2^{
        ( 4 \theta + 4 \alpha - 2 )  
      }
  \right]^{
    \! \frac{1}{2}
  }
  \left| t_2 - t_1 \right|^{ \theta }
\end{equation}
for every $ t_1, t_2 \in [0,T]$,
$ N \in \mathbb{N}$,
$ p \in [1,\infty)$
and every 
$ \alpha $,
$ \theta \in (0,\frac{1}{2}]$,
where 
$ 
  C_\star \in [0,\infty)
$
is a constant 
which depends 
on $ d, p, \alpha $ and $ \theta $ 
only 
and
where 
the stochastic process
$ 
  O^N \colon [0,T]
  \times \Omega
  \rightarrow C( [0,1]^d, \mathbb{R} 
  ) 
$ 
is defined through
\begin{equation}
\label{eq:defON}
  O^N_t(\omega)
  :=
      \sum_{ i \in 
      \{ 1, \dots, N \}^d }
      b(i) 
      \Big( 
      - \lambda_i
      \int^t_0
      e^{ 
      - \lambda_i 
      (t-s) }
      \beta^i_s(\omega) \, ds
      + \beta^i_t(\omega) 
      \Big) \cdot
      e_i
\end{equation}
for every $ t \in [0,T]$,
$ \omega \in \Omega $
and every $ N \in \mathbb{N}$.
Here
$ 
  e_i \in 
  C( [0,1]^d,\mathbb{R} ) 
$, 
$ i \in \mathbb{N}^d$,
and $ \lambda_i \in \mathbb{R} $,
$ i \in \mathbb{N}^d $,
are given in \eqref{deflambdai}.
\end{lemma}

\begin{proof}
Throughout this proof
let $ \alpha,
\theta \in (0,\frac{1}{2}]$,
$ p, N \in \mathbb{N}$
with $ p > \frac{1}{\alpha} $
and $ t_1, t_2 \in [0,T] $
with $ t_1 \leq t_2 $
be fixed.
In addition, let 
$ 
  C = 
  C_{ d,p,\alpha,\theta } \in [0,\infty) 
$
be a constant which changes
from line to line but depends 
on $ d $, $ p $, $ \alpha $
and $ \theta $ only.
We show now 
inequality~\eqref{eq:toshow1}
for these parameters 
and 
the case with a general
$ p \in [1,\infty)$
then follows from Jensen's inequality.
The definition of $ O^N $ implies
\begin{equation}
\begin{split}
&
      ( O^N_{t_2}(x) - O^N_{t_1}(x))
      - (O^N_{t_2}(y) - O^N_{t_1}(y))
\\ & =
        \sum_{ i \in \{ 1, \dots, N \}^d }
\!\!\!
        b(i) 
        \Big( 
          \int^{t_2}_{0}
          e^{-\lambda_i
          (t_2-s) } \,
          d\beta^i_s
          -
          \int^{t_1}_0
          e^{-
            \lambda_i
            (t_1-s) 
          } \,
          d\beta^i_s
        \Big) 
\cdot
        ( e_i(x) - e_i(y) )
\end{split}
\end{equation}
$ \mathbb{P} $-a.s.\ for 
every $ x,y \in [0,1]^d$.
Hence, Lemma \ref{lipbasis}
and Lemma \ref{oholder} yield
\begin{align}
&
    \mathbb{E}\Big[
    \left|
      (O^N_{t_2}(x) - O^N_{t_1}(x))
      - 
      (O^N_{t_2}(y) - O^N_{t_1}(y))
    \right|^2
  \Big]
\nonumber
\\ & =
        \sum_{ i \in \{ 1, \dots, N \}^d }
        | b(i) |^2 \,
        \mathbb{E}\!\left[
           \Big| \int^{t_2}_{0}
          e^{-\lambda_i(t_2-s) } \, d\beta^i_s
          - \int^{t_1}_0 
          e^{-\lambda_i(t_1-s)} \, d\beta^i_s
        \Big|^2 \right]
       |e_i(x) - e_i(y)|^2
\nonumber
\\ & \leq
  \sum_{ 
    i \in \{ 1, \dots, N \}^d 
  }
    | b(i) |^2 \,
    | \lambda_i |^{ (2 \theta-1) } \,
    | t_2 - t_1 |^{2 \theta}
    \left(
      2^d \pi^2
      \| i \|_2^2
      \| x - y \|_2^2
    \right)^{
      2 \alpha
    }
    \left(  
      \left| e_i(x) \right|
      + \left| e_i(y) \right|
    \right)^{
      2 \left(
        1 - 2 \alpha
      \right)
    }
\nonumber
\\ & \leq
  C
  \,
  | t_2 - t_1 |^{ 2 \theta }
  \,
  \| x - y \|_2^{4 \alpha}
  \sum\nolimits_{ i \in \{ 1, \dots, N \}^d }
  | b(i)|^2 \,
  \| i \|_2^{ (4 \theta + 4 \alpha - 2) }
\label{oholderuse1}
\end{align}
for every $ x,y \in [0,1]^d$.
Moreover, 
Lemma~\ref{oholder} gives

\begin{equation}
\label{oholderuse2}
\begin{split}
&
  \mathbb{E}\!\left[
    | O^N_{t_2}(x) - O^N_{t_1}(x)|^2
  \right]
\\ & =
  \sum_{ i \in \{ 1, \dots, N \}^d }
  | b(i)|^2 \,
  \mathbb{E}\!\left[
    \Big| 
      \int^{t_2}_{0}
      e^{-\lambda_i(t_2-s) } \,
      d\beta^i_s
      - 
      \int^{t_1}_0
      e^{- \lambda_i(t_1-s)  } \, 
      d\beta^i_s
    \Big|^2
  \right]
  |e_i(x)|^2
\\
&\leq
  C\!\!
    \sum_{ i \in \{ 1, \dots, N \}^d }
    | b(i) |^2 \,
    \| i \|_2^{ ( 4 \theta - 2 ) } \,
    | t_2 - t_1 |^{ 2 \theta }
\end{split}
\end{equation}
for every $ x \in [0,1]^d$.
In the next step the
Sobolev embeddings
in Subsections~2.2.4 
and 2.4.4
in \cite{rs96} 
yield
\begin{align*}
  &
  \mathbb{E}\big[
  \|O^N_{t_2} - O^N_{t_1}\|_{ C([0,1]^d,\mathbb{R}) }^p 
  \big]
\\
&\leq
  C 
  \int_{ (0,1)^d }
  \int_{ (0,1)^d }
  \frac{
    \mathbb{E}\big[
   |(O^N_{t_2}(x) - O^N_{t_1}(x))
- (O^N_{t_2}(y) - O^N_{t_1}(y))|^p
    \big]
  }{
    \| x - y \|_2^{ d 
        + p \alpha }
  }
  \, dx \, dy
\\ & \quad
+
  C
  \int_{ (0,1)^d }
    \mathbb{E}\big[
    | O^N_{t_2}(x) - O^N_{t_1}(x)|^p
    \big] \,
  dx
\\
&\leq
  C
  \int_{ (0,1)^d }
  \int_{ (0,1)^d }
  \frac{
    \left(
    \mathbb{E}\big[
    |(O^N_{t_2}(x) - O^N_{t_1}(x))
      - (O^N_{t_2}(y) - O^N_{t_1}(y))|^2
    \big]
    \right)^{ \! \frac{p}{2} }
  }{
    \| x - y\|_2^{ d + p \alpha
    }
  }
  dx \, dy
\\ & \quad
+
  C
  \int_{ (0,1)^d }
    \left(
    \mathbb{E}\big[
    |O^N_{t_2}(x) - O^N_{t_1}(x)|^2
    \big]
    \right)^{ \! \frac{p}{2} }
  dx
\end{align*}
and 
\eqref{oholderuse1}
and \eqref{oholderuse2}
therefore show
\begin{align*}
&
  \mathbb{E}
    \| O^N_{t_2} - O^N_{t_1}\|_{ C([0,1]^d,\mathbb{R}) }^p 
\\ & \leq
  C
  \int_{ (0,1)^d }
  \int_{ (0,1)^d }
  \frac{ 
    | t_2 - t_1 |^{ p \theta } \,
    \| x - y \|_2^{ 2 p \alpha }
  }{
    \| x - y \|_2^{  d + p \alpha  }
  }
  \, dx \, dy
  \Big(
    \sum\nolimits_{ i \in \{1,\dots,N\}^d}
    | b(i)|^2 \,
    \| i \|_2^{ 
      ( 4 \theta + 4 \alpha - 2 )
    }
  \Big)^{ \! \frac{p}{2} }
\\ & \quad
  + C
  \Big(
    \sum\nolimits_{ 
      i \in \{ 1, \dots, N \}^d 
    }
    | b(i) |^2 \,
    \| i \|_2^{ ( 4 \theta - 2 ) } \,
    | t_2 - t_1 |^{ 2 \theta }
  \Big)^{ \! \frac{p}{2} }
\\ & \leq
  C
  \Big(
    1+
    \int_{ (0,1)^d }
    \int_{ (0,1)^d }
      \| x - y \|_2^{ p \alpha - d }
    \, dx \, dy 
  \Big)
  | t_2 - t_1 |^{ p \theta }
  \Big(
    \sum_{ i \in \{1,\dots,N\}^d
    }
    | b(i) |^2
    \| i \|_2^{
        4 \theta
        + 4 \alpha - 2 
    }
  \Big)^{ \! \frac{p}{2} } .
\end{align*}
Lemma~\ref{esti2} 
hence gives 
\begin{equation}
  \left( 
    \mathbb{E}\!\left[
      \left\|
        O^N_{t_2}
        -
        O^N_{t_1}
      \right\|_{
        C( [0,1]^d, \mathbb{R} )
      }^p
    \right]
  \right)^{ \! \frac{1}{p} }
\leq
  C
  \Big(
    \sum_{ i \in 
      \left\{1,\dots,N\right\}^d
    }
    | b(i) |^2 \,
    \| i \|_2^{
      (
        4 \theta
        + 4 \alpha - 2 
      ) 
    }
  \Big)^{ \! \frac{1}{2} }
  \left| 
    t_2 - t_1 
  \right|^{
    \theta 
  } 
\end{equation}
and this completes the
proof of Lemma~\ref{oholder2}.
\end{proof}
%

%
%
%
%
%
%
%
%
%

\begin{lemma}
\label{oGalerkin}
Let $ d \in \mathbb{N}$,
let 
$ 
  \beta^i \colon [0,T] 
  \times \Omega \rightarrow
  \mathbb{R} 
$, 
$ 
  i \in \mathbb{N}^d
$,
be a family 
of independent 
standard Brownian motions
and 
let 
$ 
  b \colon \mathbb{N}^d
  \rightarrow \mathbb{R} 
$
be an arbitrary function.
Then
\begin{equation}
\label{eq:toshow2}
  \Big(
  \mathbb{E}\Big[
    \sup_{0\leq t \leq T}
    \sup_{ x \in [0,1]^d }
    \left| O^N_{t}(x) 
    - O^M_{t}(x) \right|^p
  \Big]
  \Big)^{ \! \frac{2}{p} }
  \leq 
  C_\star\!\!\!\!\!\!\!\!
  \sum_{ 
      i \in \left\{1,\dots,N
      \right\}^d 
      \backslash
      \{1,\dots,M\}^d 
  }\!\!\!\!\!\!\!\!
      | b(i) |^2 \,
      \| i \|_2^{
        \left( 8 \alpha - 2 
      \right)}
\end{equation}
for every 
$ 
  N, M \in \mathbb{N}
$
with 
$ 
  N \geq M 
$,
every
$ p \in [1,\infty) $
and every
$
  \alpha \in \left(0,\frac{1}{2}
  \right)
$,
where 
$ 
  C_\star \in [0,\infty) 
$
is a constant which depends 
on $ d $, $ p $, $ \alpha $ and $T$
only
and where 
$ 
  O^N \colon [0,T]
  \times \Omega
  \rightarrow 
  C( 
    [0,1]^d, \mathbb{R} 
  ) 
$,
$ N \in \mathbb{N} $,
are stochastic
processes defined
through \eqref{eq:defON}.
\end{lemma}

\begin{proof}
Throughout this proof
let $ \alpha \in (0,\frac{1}{2})$
and $ p, N, M \in \mathbb{N}$
with $ p > \frac{1}{\alpha} $ and
$N \geq M $
be fixed.
In addition, let 
$ 
  C = C_{ d, p, \alpha, T } 
  \in [0,\infty)
$
be a constant, which changes
from line to line but which depends on
$ d $, $ p $, $ \alpha $ and 
$ T $
only.
As in the proof of Lemma \ref{oholder2},
we show now 
inequality~\eqref{eq:toshow2}
for these parameters and 
the case with a general
$ p \in [1,\infty)$
then follows from 
Jensen's inequality.
We use the factorization method
(see 
\cite{DaPratoKwapienZabczyk1988} 
and, e.g., Section 5.3
in \cite{dz92} and Section~5 
in \cite{bms01})
to show \eqref{eq:toshow2}.
For this let
$
  Y^{N,M} \colon 
  [0,T] \times \Omega \rightarrow   
  C( [0,1]^d,\mathbb{R} ) 
$
be a stochastic processes
with continuous sample
paths given by
\begin{equation}
  Y^{N,M}_t
  =
  \sum_{ 
      i \in \{1,\dots,N\}^d
      \backslash 
      \{1,\dots,M\}^d
  }
      b(i) 
        \int^t_0
        \left( t - s 
        \right)^{-\alpha}
        e^{ 
        - \lambda_i (t-s) 
        } 
        \,
        d\beta^i_s
\cdot
      e_i
\end{equation}
$ \mathbb{P} $-a.s.\ for 
every $ t \in [0,T]$.
By using 
Kolmogorov's theorem
(see, e.g., Theorem~3.3 
in \cite{dz92}),
one can check in a straightforward
way that the
stochastic processes 
$
  \int^t_0
  ( t - s )^{-\alpha}
  \,
  e^{ 
    \lambda_i (t-s) 
  }
  \,
  d\beta^i_s
$,
$
  t \in [0,T]
$,
$
  i \in \mathbb{N}^d
$,
indeed have modifications 
with continuous sample paths.
The key idea of the factorization
method is then to make use of
the identity
\begin{equation}
\label{eq:factorization}
  O^{N}_t - O^M_t
  =
  \frac{ \sin(\pi \alpha) }{\pi }
  \int^t_0 
    ( t - s)^{ ( \alpha - 1 ) } \,
    S_{ t - s } \, Y^{N,M}_s \, 
  ds
\end{equation}
$ \mathbb{P} $-a.s.\ for 
every $ t \in [0,T] $
(see, e.g., equation~(5.18)
in Section~5.3
in \cite{dz92}).
More precisely, 
combining \eqref{eq:factorization}, 
the well known fact
$ 
  \sup_{ 0 \leq t \leq T }
  \| S_t \|_{ 
    L( 
      C( [0,1]^d, \mathbb{R} ) 
    )
  } \leq 1 
$
(see, e.g.,
Lemma~6 in \cite{j09b})
and H\"{o}lder's inequality
gives 
\begin{equation}
\label{useoGal}
\begin{split}
&
  \mathbb{E}
    \sup_{0 \leq t \leq T}
    \left\| 
      O^{N}_t - O^M_t
    \right\|_{
      C([0,1]^d,\mathbb{R})
    }^p
\\ & =
  \mathbb{E}
  \sup_{0 \leq t \leq T}
  \Big\| 
    \frac{ \sin(\pi \alpha) }{\pi }
    \int^t_0 
    ( t - s)^{ ( \alpha - 1 ) } \,
    S_{ t - s } \, Y^{N,M}_s \, ds
  \Big\|_{
    C([0,1]^d,\mathbb{R})
  }^p
\\ & \leq
  \mathbb{E}
  \sup_{0 \leq t \leq T}
  \Big|
    \int^t_0 
    ( t - s )^{( \alpha - 1 )} \,
    \| Y^{N,M}_s \|_{ 
      C([0,1]^d,\mathbb{R})
    } \,
  ds
  \Big|^p
\\ & \leq
  C
  \int^T_0 
  \mathbb{E}\!\left[
    \| Y^{N,M}_s \|_{
      C([0,1]^d,\mathbb{R})
    }^p 
  \right]
  ds\;. 
\end{split}
\end{equation}
Hence, it remains to bound 
$\| Y^{N,M}_s \|_{C([0,1]^d,\mathbb{R})}$
in (\ref{useoGal}).
For this, 
denote 
$
  \mathcal{I}_{N}
  :=
  \{ 1, 2, \dots, N\}^d
$ 
and
$ 
  \mathcal{I}_{M} :=
  \{1, 2, \dots, M\}^d
$.
Lemma~\ref{lipbasis}
then implies
\begin{equation}
\label{oGaluse1}
\begin{split}
&
  \mathbb{E}\big[
    |  Y^{N,M}_t(x)-Y^{N,M}_t(y)|^2
  \big]
\\ & =
  \mathbb{E}\Big[
    \Big|
    \sum_{   
      i \in \mathcal{I}_N
      \backslash \mathcal{I}_M 
    }
      b(i) 
      \int^t_0
      ( t - s )^{-\alpha} \,
      e^{ - \lambda_i (t-s) } 
      \, d\beta^i_s
\cdot
      ( e_i(x) - e_i(y) )
    \Big|^2
  \Big]
\\ & =
  \sum_{ 
    i \in \mathcal{I}_N
    \backslash \mathcal{I}_M 
  }
  | b(i) |^2 
  \;
  \mathbb{E}\Big[ 
    \Big|
      \int^t_0
        ( t - s )^{-\alpha} \,
        e^{ - \lambda_i (t-s) } \, 
      d\beta^i_s
    \Big|^2
  \Big]
\cdot
  | e_i(x) - e_i(y) |^2
\\ & =
  \sum_{ 
    i \in \mathcal{I}_N
    \backslash \mathcal{I}_M 
  }
  | b(i)|^2
    \int^t_0 
      s^{-2\alpha} \,
      e^{  - 2 \lambda_i s } \,
    ds
\cdot
	| e_i(x) - e_i(y) |^{ 4 \alpha } \,
  | e_i(x) - e_i(y) |^{ (2 - 4 \alpha) }
\\ & \leq
  C
  \sum_{ 
    i \in \mathcal{I}_N
    \backslash \mathcal{I}_M 
  }
  | b(i) |^2 \,
  \| i \|_2^{
    ( 8 \alpha - 2 )  
  }
  \left\| 
    x - y 
  \right\|_2^{ 4 \alpha }
\end{split}
\end{equation}	
for every $ t \in [0,T]$
and every $ x,y \in [0,1]^d$.
In addition, note that
\begin{equation}
\label{oGaluse2}
\begin{split}
&
  \mathbb{E}\!\left[
    | Y^{N,M}_t(x) |^2
  \right]
=
  \mathbb{E}\Big[
  \Big|
    \sum_{ 
      i \in \mathcal{I}_N
      \backslash \mathcal{I}_M 
    }
    b(i) 
      \int^t_0
      ( t - s)^{ - \alpha } \,
      e^{ - \lambda_i (t-s) } \,
      d\beta^i_s
\cdot
    e_i(x)
  \Big|^2
  \Big]
\\
&=
  \sum_{ i \in \mathcal{I}_N
  \backslash \mathcal{I}_M }
    | b(i)|^2 \;
    \mathbb{E}\Big[
      \Big|
        \int^t_0
        ( t - s )^{ - \alpha } \,
        e^{ 
          - \lambda_i (t-s) 
        } 
        \,
        d\beta^i_s
      \Big|^2
    \Big]
    | e_i(x)|^2
\\ & =
  \sum_{ 
    i \in \mathcal{I}_N
    \backslash \mathcal{I}_M 
  }
  | b(i)|^2
    \int^{ 2 t \lambda_i }_0 
    s^{ - 2 \alpha } \,
    e^{ - s } \,
    ds
  ( 2 \lambda_i )^{ ( 2 \alpha - 1 ) } \,
  | e_i(x)|^2
  \leq
  C \!\!\!
  \sum_{ 
    i \in \mathcal{I}_N
    \backslash \mathcal{I}_M 
  }
  | b(i) |^2 \,
  \| i \|_2^{ 
    ( 8 \alpha - 2 )
  }
\end{split}
\end{equation}
for every $ t \in [0,T]$ and
every $ x, y \in [0,1]^d $.
In the next step the
Sobolev embeddings
in Subsections~2.2.4 
and 2.4.4
in \cite{rs96} give
\begin{equation}
\begin{split}
  \sup_{ 0 \leq t \leq T }
  \mathbb{E}
    \| 
      Y^{ N, M }_t
    \|_{ C( [0,1]^d, \mathbb{R} ) }^p
& \leq
  C
  \sup_{ 0 \leq t \leq T }
  \int_{
    ( 0, 1 )^d
  }
  \int_{
    ( 0, 1 )^d
  }
  \frac{
    \big(
      \mathbb{E}\big[
        | Y^{N,M}_t(x) - Y^{N,M}_t(y) |^2
      \big]
    \big)^{ { p/2 } }
  }{
    \| x - y \|_2^{ d + p \alpha  }
  }
  \, dx \, dy
\\ & \quad
  + C
  \sup_{ 0 \leq t \leq T }
  \int_{
    (0,1)^d
  }
    \left(
      \mathbb{E}\big[
        | Y^{N,M}_t(x)|^2
      \big]
    \right)^{ \! \frac{ p }{ 2 } }
  dx
\end{split}
\end{equation}
and 
\eqref{oGaluse1}, \eqref{oGaluse2} 
and Lemma \ref{esti2}
therefore imply
\begin{equation}
\begin{split}
&
  \sup_{ 0 \leq t \leq T }
    \mathbb{E}\big[
      \| 
        Y^{N,M}_t
      \|_{ C( [0,1]^d, \mathbb{R}) }^p
    \big]
\\ & \leq
    C
    \int_{
      ( 0, 1 )^d
    }
    \int_{
      ( 0, 1 )^d
    }
    \frac{
      \left(
        \sum_{ 
          i \in \mathcal{I}_N
          \backslash \mathcal{I}_M 
        }
        | b(i)|^2 \,
        \| i \|_2^{ ( 8 \alpha - 2 ) } \,
        \| x - y \|_2^{ 4 \alpha }
      \right)^{ \frac{ p }{ 2 } }
    }{
      \| x - y \|^{ ( d + p \alpha ) }
    }
    \, dx \, dy
\\ & \quad
  + C 
    \Big(\sum_{ 
      i \in \mathcal{I}_N
      \backslash \mathcal{I}_M 
    }
    | b(i) |^2 \,
    \| i \|_2^{ ( 8 \alpha - 2 ) }\Big)^{p/2 }
  \leq 
  C
  \Big(
    \sum_{ 
      i \in \mathcal{I}_N
      \backslash \mathcal{I}_M 
    }
    | b(i) |^2 \,
    \| i \|_2^{ ( 8 \alpha - 2 ) } 
  \Big)^{p/2 } .
\end{split}
\end{equation}
This and inequality~\eqref{useoGal}
then show \eqref{eq:toshow2}.
The proof of Lemma~\ref{oGalerkin}
is thus completed.
\end{proof}
%

%
%
%
%


%
%
%

\begin{proof}[Proof
of Lemma~\ref{constructOO}]
Throughout this proof
let
$ 
  O^N \colon [0,T] \times 
  \Omega
  \rightarrow 
  C( [0,1]^d, \mathbb{R} ) 
$,
$ 
  N \in \mathbb{N} 
$,
be a sequence of
stochastic processes
defined through
\eqref{eq:defON}.
Next note that
Lemma~\ref{oGalerkin} implies

\begin{equation}
\begin{split}
  \left( 
    \mathbb{E}\!\left[
      \sup_{0 \leq t \leq T}
      \| 
        O^{N}_t - O^M_t
      \|_{
        C([0,1]^d,  
        \mathbb{R})
      }^p
    \right]
  \right)^{ \! \frac{ 1 }{ p } }
 & \leq
  C
  \Big(
    \sum_{ 
      i \in \mathbb{N}^d
      \backslash
      \{1,\dots,M\}^d
    }
    | b(i)|^2 \,
    \| i\|_2^{ ( 8 \alpha - 2 ) }
  \Big)^{ \! \frac{ 1 }{ 2 } }
\\ & \leq
  C
  \Big(
    \sum_{ i \in \mathbb{N}^d} 
    | b(i) |^2 \,
    \| i \|_2^{ ( 2 \rho - 2 ) }
  \Big)^{ \! \frac{1}{2} }
  M^{ ( 4 \alpha - \rho ) }
\end{split}
\end{equation}
for every $ N, M \in \mathbb{N}$
with $N \geq M$,
every $ p \in [1,\infty)$
and every 
$
  \alpha \in 
  (0,\min(\frac{1}{2},\frac{\rho}{4}))
$
where
$ 
  C
  \in [0,\infty) 
$
is a constant
which depends on $ d $,
$ p $, $ \alpha $ and $ T $
only.
This, in particular, 
gives that 
$ 
  O^N \colon [0,T] \times
  \Omega \rightarrow 
  C( [0,1]^d, \mathbb{R} )
$,
$ 
  N \in \mathbb{N} 
$,
is a Cauchy
sequence in 
$
  L^p( \Omega;
    C( [0,T], 
      C( [0,1]^d, \mathbb{R} ) 
    )
  )
$.
Hence,
there exists a stochastic process
$ 
  \tilde{O} \colon [0,T] \times
  \Omega \rightarrow 
  C( [0,1]^d, \mathbb{R} ) 
$ 
with
continuous sample paths
which satisfies
\begin{equation}
   \Big( 
    \mathbb{E}\Big[
      \sup_{ 0 \leq t \leq T } 
      \| 
        \tilde{O}_t - O^N_t 
      \|_{ C( [0,1]^d, \mathbb{R} ) }^p
    \Big]
   \Big)^{ \! \frac{ 1 }{ p } }
  \leq
  C
  \Big(
    \sum_{ 
      i \in \mathbb{N}^d
    }
    | b(i) |^2 \,
    \| i \|_2^{ ( 2 \rho - 2 ) } 
  \Big)^{ \! \frac{ 1 }{ 2 } }
  N^{ ( 4 \alpha - \rho ) }
\end{equation}
for every 
$ 
  N \in \mathbb{N}
$,
every $ p \in [1,\infty)$ and
every 
$ 
  \alpha \in 
  (0,
    \min( \frac{1}{2}, \frac{\rho}{4} )
  )
$.
Therefore, we have
\begin{equation}
  \sup_{N \in \mathbb{N}}
  \left\{
    N^{ \gamma }
    \left( 
      \mathbb{E}\!\left[
        \sup_{ 0 \leq t \leq T }
        \|
          \tilde{O}_t - O^N_t
        \|_{ C( [0,1]^d, \mathbb{R}) }^p
      \right]
    \right)^{ \! \frac{ 1 }{ p } }  
  \right\}
  < \infty
\end{equation}
for every $ \gamma \in (0,\rho) $
and every $ p \in [1,\infty)$.
This implies
\begin{equation}
  \mathbb{P}\!\left[ \,
    \sup_{N \in \mathbb{N}}
    \left(
      N^\gamma
      \sup_{ 0 \leq t \leq T }
      \|
        \tilde{O}_t - O^N_t
      \|_{ C([0,1]^d,\mathbb{R}) } 
    \right)
    < \infty \,
  \right]
  = 1 
\end{equation}
for every $ \gamma \in (0,\rho) $
due to Lemma~2.1 in \cite{kn10}.
This yields
\begin{equation}
  \mathbb{P}\!\left[\;
    \forall
    \;
    \gamma \in (0,\rho) 
    \colon
    \sup_{N \in \mathbb{N}}\,
    \sup_{0 \leq t \leq T}
    \left(
      N^\gamma \,
      \|\tilde{O}_t - O^N_t
      \|_{C([0,1]^d,\mathbb{R})} 
    \right)
    < \infty
    \;
  \right]
  = 1 
\end{equation}
and hence, we obtain that
\begin{align}
\label{oGalusea}
&
  \mathbb{P}\!\left[\;
    \lim_{N\rightarrow \infty}
    \,
    \sup_{0 \leq t \leq T}
      \| 
        \tilde{O}_t -  O^N_t
      \|_{C([0,1]^d,\mathbb{R})}
    = 0
    \;
  \right]
  = 1 
  \qquad
  \text{and}
\\ 
\label{oGalusea2}
&
  \mathbb{P}\!\left[\;
    \forall
    \;
    \gamma \in (0,\rho) 
    \colon
    \sup_{N \in \mathbb{N}}\,
    \sup_{0 \leq t \leq T}
    \left(
      N^\gamma \,
      \|\tilde{O}_t - P_N( \tilde{O}_t )
      \|_{C([0,1]^d,\mathbb{R})} 
    \right)
    < \infty
    \;
  \right]
  = 1 .
\end{align}
In addition,
Lemma \ref{oholder2}
gives
\begin{align*}
  \Big( 
    \mathbb{E}\Big[
      \| 
        O^N_{t_2} - O^N_{t_1}
      \|_{ 
        C( [0,1]^d, \mathbb{R} ) 
      }^p
    \Big]
  \Big)^{ \! \frac{ 1 }{ p } }
& \leq
  \tilde{C}_{ d, p, \rho, \theta }
  \Big(
    \sum_{i \in \{1,\dots,N\}^d }
    | b(i) |^2 \,
    \| i \|_2^{ 
        4 \theta
        + 4( \frac{\rho}{2} - \theta) - 2 
    }
  \Big)^{ \! \frac{1}{2} }
  \left| t_2 - t_1 \right|^{
    \theta
  }
\\ & \leq
  \tilde{C}_{ d, p, \rho, \theta }
  \Big(
    \sum_{i \in \mathbb{N}^d }
    | b(i) |^2 \,
    \| i \|_2^{ ( 2 \rho - 2 ) }
  \Big)^{ 
    \! \frac{1}{2} 
  }
  \left| 
    t_2 - t_1 
  \right|^{
    \theta
  }
\end{align*}
for every $ t_1, t_2 \in [0,T]$,
$N \in \mathbb{N}$,
$ p \in [1,\infty) $
and every 
$ 
  \theta \in (0,\frac{\rho}{2})
  \cap
  [0, \frac{1}{2} ]
$
where 
$ 
  \tilde{C}_{ d, p, \rho, \theta } 
  \in [0,\infty)
$
is a constant
which depends
on $ d $, $ p $, $ \rho $
and $ \theta $ only.
This shows
\begin{equation}
  \left( 
    \mathbb{E}\!\left[
      \| 
        \tilde{O}_{t_2} - \tilde{O}_{t_1}
      \|_{ C([0,1]^d,\mathbb{R}) }^p
    \right]
  \right)^{ \! \frac{ 1 }{ p } }
\leq
  \tilde{C}_{ d, p, \rho, \theta }
  \Big(
    \sum_{ i \in \mathbb{N}^d }
    | b(i) |^2 \,
    \| i \|_2^{ 2 \rho - 2 }
  \Big)^{ 
    \! \frac{1}{2} 
  }
 | t_2 - t_1 |^{\theta}
\end{equation}
for every $ t_1, t_2 \in [0,T]$,
$ p \in [1,\infty) $
and every $ \theta \in 
\left(0,\frac{\rho}{2}\right)$,
$\theta \leq \frac{1}{2}$.
Kolmogorov's theorem 
(see, e.g., 
Theorem~3.3 
in \cite{dz92}) 
hence yields
\begin{equation}
  \mathbb{P}\!\left[
    \sup_{ 0 \leq t_1 < t_2 \leq T }
    \frac{
      \| 
        \tilde{O}_{t_2} - \tilde{O}_{t_1}
      \|_{ C([0,1]^d,\mathbb{R}) }
    }{
      | t_2 - t_1 |^{ \theta }
    }
    < \infty
  \right] = 1
\end{equation}
for every 
$ 
  \theta \in 
  \left(
    0,
    \min\left(
      \frac{ 1 }{ 2 }, \frac{ \rho }{ 2 }
    \right)
  \right)
$.
This implies
\begin{equation}
\label{oGaluseb}
  \mathbb{P}\!\left[
    \,
    \forall
    \;
    \theta \in 
    ( 
      0, \min\{\tfrac{1}{2},
      \tfrac{\rho}{2}\} 
    )
    \colon
    \sup_{ 0 \leq t_1 < t_2 \leq T }
    \frac{
      \| 
        \tilde{O}_{ t_2 } - \tilde{O}_{ t_1 }
      \|_{
        C( [0,1]^d, \mathbb{R})
      }
    }{
      \left| t_2 - t_1 \right|^\theta
    }
    < \infty
    \,
  \right] = 1 .
\end{equation}
Combining
\eqref{oGalusea2} and
\eqref{oGaluseb}
shows 
the existence 
of a stochastic
process
$ 
  O \colon [0,T] \times \Omega 
  \rightarrow 
  C( [0,1]^d,\mathbb{R} ) 
$
with continuous sample paths
which is indistinguishable from 
$ 
  \tilde{O}
$, i.e.,
$
  \mathbb{P}\big[
    \, \forall 
    \, t \in [0,T] 
    \colon
    O_t = \tilde{O}_t 
    \,
  \big]
  = 1 
$
and which satisfies
$
  \sup_{0 \leq t_1 < t_2 \leq T}
  \frac{
    \| 
      O_{t_2}(\omega)
      - O_{t_1}(\omega)
    \|_{
      C( [0,1]^d, \mathbb{R} )
    }
  }{
    \left| t_2 - t_1 \right|^{ \theta }
  }
$
$
  < \infty
$
and
$
  \sup_{N \in \mathbb{N}}
  \sup_{0 \leq t \leq T}
  \left(
    N^{ \gamma }
    \left\| 
      O_t(\omega) -
      P_N( O_t(\omega) )
    \right\|_{
      C( [0,1]^d, \mathbb{R} )
    }
  \right)
  < \infty
$
for every 
$ 
  \omega \in \Omega 
$,
every 
$ 
  \theta \in 
  \left(0,
    \min\left(\frac{1}{2},\frac{\rho}{2}\right)
  \right)
$ 
and
every 
$ 
  \gamma \in (0,\rho)
$.
The proof of 
Lemma~\ref{constructOO}
is thus completed.
\end{proof}
%

%
%
%
%
%
\subsection{Proofs
for Subsection~\ref{stochburger}}

\subsubsection{Proof of
Lemma~\ref{BurgerSG}}
\label{sec:burgerSG}
%
%
%
%
\begin{proof}
First, note that
\begin{align*}
&
  \sum^{\infty}_{n=1}
    n^{ 2 + 2 \gamma }
    \, 
    e^{ - 2 n^2 \pi^2 t }
\leq
  \int^{\infty}_0 ( x + 1 )^{  2 + 2 \gamma  }
  e^{ - 2 x^2 \pi^2 t }
  \, dx
\\&\leq
  \int^{\infty}_0
  8 \big( x^{ 2 + 2 \gamma } + 1 )
  e^{ - 2 x^2 \pi^2 t }
  \, dx
=
  \frac{ 1 }{
    2 \pi \sqrt{t}
  }
  \int^{\infty}_0
  8 \Big( 
    \frac{ x^{ 2 + 2 \gamma  } 
    }{ 
      ( 2 \pi \sqrt{t} )^{ 2 + 2 \gamma } 
    } + 1 \Big)
 e^{ - \frac{ x^2 }{ 2 } }
  dx
\\ & \leq
  \frac{4}{
    \pi \sqrt{t}
  }
  \int^{\infty}_0
    \Big( 
    \frac{ x^{ 2 + 2 \gamma } 
    }{ 
      t^{  1 + \gamma  } } + 1 \Big)
  e^{ - \frac{ x^2 }{ 2 } }
  dx
\leq
  \frac{4}{
    \pi \sqrt{t}
  }
  \int^{ \infty }_0
    \left( 
    \frac{ 
       x^4 + 1 
    }{  
      t^{ 1 + \gamma  } } + 
      \frac{ T^{ 1 + \gamma } }{ 
        t^{  1 + \gamma  } 
      } \right)
  e^{ - \frac{ x^2 }{ 2 } } 
  dx
\\&\leq
  \frac{4 \sqrt{2 \pi} }{
    t^{  \frac{3}{2} + \gamma } \pi
  }
  \int_{\mathbb{R}}
    \frac{ 
      x^4 +   
      T^2 + 2
    }{ \sqrt{ 2 \pi } }
    \,	
    e^{ - \frac{ x^2 }{ 2 } }
  dx
\leq
  \frac{
    4 ( T^2 + 5 ) 
  }{
    t^{  \frac{3}{2} + \gamma
    }
  }
\end{align*}
for every $ t \in (0,T] $
and every $ \gamma \in [0,\frac{1}{2})$.
The identity
$
  \| w \|_{ H^{-1} }
$
$
  =
$
$
  \sum^{\infty}_{n=1}
  n^{ - 2 } \pi^{ - 2 }
$
$
  | 
    w( \sqrt{2} \sin( n \pi \cdot )
    )    
  |^2 
$
for every 
$ 
  w \in 
  H^{-1}( 
    (0,1), \mathbb{R}
  ) 
$
hence gives
\begin{align*}
&
  \sup_{ 0 \leq x \leq 1 }
  \Big(
  \sum^{\infty}_{n=N}
  2 \cdot e^{ - n^2 \pi^2 t }
  \cdot |w(\sin( n \pi ( \cdot ) ) )|
  \cdot
  | \sin( n \pi x )|
  \Big)
\\&\leq 
  \pi \sqrt{2}
  \sum^{\infty}_{n=N}
  n \, e^{ - n^2 \pi^2 t }
  \,
  \frac{ 
    | w( \sqrt{2} \sin( n \pi ( \cdot ) ) )| 
  }{ 
    n \pi 
  }
\\&\leq
  \pi \sqrt{2} 
  \Big(
    \sum^{\infty}_{n=N}
    n^2 \, e^{ - 2 n^2 \pi^2 t }
  \Big)^{ \! \frac{1}{2} }
  \| w \|_{ H^{-1} }
\leq
  \pi \sqrt{2} \,N^{-\gamma}
  \Big(
  \sum^{\infty}_{n=N}
    n^{ (2 + 2 \gamma) } \,
    e^{ - 2 n^2 \pi^2 t }
  \Big)^{ \! \frac{1}{2} }
  \| w \|_{ H^{-1} }
\\&\leq
  \pi \sqrt{2} \,N^{-\gamma}
  \left(
    4 \left( T^2 + 5 \right)
    t^{ 
      - \left( \frac{3}{2} + \gamma 
      \right) 
    }
  \right)^{ \frac{1}{2} }
  \| w \|_{ H^{-1} }
\leq
 10 \left( T + 3 \right) 
      t^{ -( \frac{3}{4} + \frac{ \gamma }{2} ) }
      N^{-\gamma}
  \,
  \| w \|_{ H^{-1} }
\end{align*}
for every 
$ 
  w \in 
  H^{-1}( 
    (0,1), \mathbb{R}
  ) 
$, 
$ N \in \mathbb{N} $,
$ 
  \gamma \in [0,\frac{1}{2}) 
$
and every 
$ 
  t \in (0,T] 
$.
This implies 
$
  \| 
    S_t( w )
  \|_{ C( [0,1], \mathbb{R}) }
\leq
  10 \, ( T + 3 ) \,
  t^{ - \frac{3}{4} } 
  \, 
  \| w\|_{ H^{-1} }
$
and
\begin{equation}
  \|  
    S_t(w) - P_N( S_t( w ) )
  \|_{ 
    C( [0,1], \mathbb{R})
  }
\leq
  \frac{ 
    10 \left( T + 3 \right) 
    \| w \|_{ H^{-1} }
  }{
    t^{ 
      \left( \frac{3}{4} + \frac{\gamma}{2} 
      \right) 
    } 
    \left( N + 1 \right)^\gamma
  }
\leq
  \frac{ 
    10 \left( T + 3 \right) 
    \| w \|_{ H^{-1} }
  }{
    t^{ 
      \left( 
        \frac{3}{4} + \frac{\gamma}{2}
      \right)  
    }
    N^{ \gamma }
  }
\end{equation}
for every 
$ t \in (0,T] $,
$ w \in 
H^{-1}\left( (0,1), \mathbb{R} \right) $,
$ \gamma \in [0,\frac{1}{2}) $
and every $ N \in \mathbb{N} $.
Therefore, we finally obtain
$
  \sup_{0 < t \leq T }
  \big(
  t^{ \frac{3}{4} }
  \left\| S_t \right\|_{ 
    L( H^{-1}( (0,1), \mathbb{R} ), \,
    C( [0,1], \mathbb{R}) ) 
  }
  \big)
  < \infty 
$
and
\begin{equation}
  \sup_{N \in \mathbb{N}}
  \sup_{0 < t \leq T }
  \Big(
  t^{ 
    \left( 
      \frac{3}{4} + \frac{\gamma}{2} 
    \right)
  }
  N^\gamma
  \left\| S_t - P_N S_t \right\|_{ 
    L( H^{-1}( (0,1), \mathbb{R} ), \,
    C( [0,1], \mathbb{R}) ) 
  }
  \Big)
  < \infty 
\end{equation}
for every $ \gamma \in [0,\frac{1}{2}) $.
The proof of 
Lemma~\ref{BurgerSG}
is thus completed.
\end{proof}
%

%
%
\subsubsection{Proof of 
Lemma~\ref{BurgerSol}}
\label{sec:BurgerSol}
%
%
%

In the proof of Lemma~\ref{BurgerSol}
the following well known
estimates for the analytic
semigroup generated by
the Laplacian are used.
Their proofs can, e.g., be found
in Lemma~5.8 
in \cite{bj09}.

\begin{lemma}
\label{lemmadditional}
Let 
$ 
  S \colon (0,T] \rightarrow 
  L\big( H^{-1} ((0,1), \mathbb{R} ), 
  C ( [0,1], \mathbb{R}) \big) 
$ 
be given by Lemma~\ref{BurgerSG} 
and 
let 
$ 
  P_N \colon C ( [0,1], \mathbb{R}) 
  \rightarrow C ( [0,1], \mathbb{R})
$, 
$ N \in \mathbb{N} $, 
be given by \eqref{projoperator}.
Then
\begin{equation*}
\| P_N S_t \|_{ 
    L( 
      L^2( (0,1), \mathbb{R} ), 
      C( [0,1], \mathbb{R} )  
    ) 
  } 
  \leq t^{-\frac{1}{4}} \, , \qquad
  \left\| 
    P_N S_t 
  \right\|_{ L( 
    L^2( (0,1), \mathbb{R} ), 
    L^4( (0,1), \mathbb{R} ) 
  ) } 
  \leq t^{-\frac{1}{8}}\;,
\end{equation*}
\begin{equation*}
 \| S_t \|_{ L( 
    H^{-1}( (0,1), \mathbb{R} ), 
    L^2( (0,1), \mathbb{R} ) 
  ) } 
  \leq t^{-\frac{1}{2}} \,
  \quad\text{and}\quad
  \|  S_t (v') \|_{L^2}
  \leq 
  4 \, (t+1) \, 
  t^{ - \frac{3}{4} } \,
  \| v \|_{L^1}
\end{equation*}
for every $ t \in (0, T], N \in \mathbb{N} $ and every $ v \in C^1 ( [0,1], \mathbb{R} ) $.
\end{lemma}

In addition to 
Lemma~\ref{lemmadditional},
the following elementary global coercivity 
estimate for Burgers equation
is used in the proof
of Lemma~\ref{BurgerSol}
(see also Lemma~3.1 
in \cite{GdP-AD-RT:94}).

\begin{lemma}
\label{lemburger}
Let 
$ 
  F \colon C( [0,1], \mathbb{R}) 
  \rightarrow 
  H^{-1}( (0,1), \mathbb{R} ) 
$ 
be given by 
Lemma~\ref{BurgerNonl}.
Then
\begin{equation}
  \left< v, v'' + F(v+w) \right>_{L^2}
  \leq  
  2 \,c^2  
  \| v \|^2_{L^2}
  \| w \|^2_{ C([0,1] , \mathbb{R}) } 
  + 
  2 \, c^2 
  \| w \|^4_{ C([0,1] , \mathbb{R}) }
\end{equation}
for all twice
continuously differentiable functions
$ 
  v \colon [0,1] 
  \rightarrow \mathbb{R} 
$
and 
$ 
  w \colon [0,1] \rightarrow 
  \mathbb{R} 
$
with 
$ v(0) = v(1) = 0 $ 
and where 
$ c \in \mathbb{R} $ 
is used in Lemma~\ref{BurgerNonl}.
\end{lemma}

\begin{proof}
Integration by parts and 
the identity
$
  \int_0^1 v'(x) \, | v(x) |^2 \, dx=0
$
imply
\begin{equation}
\begin{split}
  \left< v, F(v + w) \right>_{L^2} 
&   = 
  - \, 2 c 
  \int_0^1 
    v'(x) \cdot v(x) \cdot w(x) \, 
  dx
  - 
  c 
  \int_0^1 
    v'(x) \cdot | w(x) |^2 \, 
  dx
\\&\leq
   2 | c |
  \left(
    \| v \|_{L^2} \cdot
    \| w \|_{ C([0,1] , \mathbb{R}) } 
    + 
     \| w \|^2_{ C([0,1] , \mathbb{R}) }
  \right) \cdot
    \| v' \|_{L^2}
\\ & \leq
   2 c^2 \| v \|^2_{L^2}
    \| w \|^2_{ C([0,1] , \mathbb{R}) } 
  +
      2 c^2 
    \| w \|_{ C([0,1] , \mathbb{R}) }^4
  +
    \| v' \|_{L^2}^2
\end{split}
\end{equation}
and therefore using integration by parts
\[
   \left< v, v'' + F(v+w) \right>_{L^2}
  \leq  
  2 c^2\| v \|^2_{L^2}
      \| w \|^2_{ C([0,1] , \mathbb{R}) } 
 + 
     2 c^2 
     \| w\|^4_{ C( [ 0,1 ] , \mathbb{R}) }
\]
for all twice continuously 
differentiable 
functions 
$
  v, w \colon [0,1] \rightarrow \mathbb R
$
with 
$ 
  v(0) = v(1) = 0
$.
The proof of Lemma~\ref{lemburger}
is thus completed.
\end{proof}

In the next step note that
Lemma~\ref{BurgerSol}
follows by combining 
the next lemma 
(see also Lemma~3.1 
in 
\cite{GdP-AD-RT:94}
for a related result)
and 
a standard fix point argument
(see also Theorem~3.2
in 
\cite{GdP-AD-RT:94}).

\begin{lemma}
\label{lemproof}
Let $ \tau \in (0,T], N \in \mathbb{N} $ 
and let 
$
  x_N \colon [0,\tau] \rightarrow 
  P_N( C([0,1], \mathbb{R}) ) 
$ 
and 
$ 
  o_N \colon 
  [0,\tau] \rightarrow 
  P_N( C([0,1], \mathbb{R}) ) 
$ 
be two continuous functions
which satisfy
$
  x_N(t) = \int_{0}^{t}  P_N \, S_{ t - s } 
  \, F ( x_N(s) ) \, ds + o_N(t)
$
for every $ t \in [0, \tau] $. 
Then
\begin{equation}
  \sup_{0 \leq t \leq \tau}
  \| x_N(t) \|_{C ( [0,1], \mathbb{R})}
  \leq 
  \exp
  \Big(
    24 ( c^2 + 1) ( T + 1 )
    \Big(
      \sup_{0 \leq t \leq \tau}
      \| o_N(t)\|^2_{C ( [0,1], \mathbb{R} )}   
      + 1
    \Big)
  \Big) ,
\end{equation}
where $ c \in \mathbb{R} $ is
used in Lemma \ref{BurgerNonl}.
\end{lemma}

\begin{proof}
First, note that 
the definition of
$ 
  S \colon (0,T]
  \rightarrow L( W, V )
$
in Lemma~\ref{BurgerSG}
implies
\begin{equation}
\label{lem7proof}
    S_t v - v
  = 
    \int_0^t S_s ( v'' ) \, ds  
  \qquad
    \text{and}
  \qquad  
    \|  
      ( S_t - I ) v 
    \|_{
      C( [0,1],\mathbb{R})
    }
  \leq 
    t \cdot
    \| 
      v'' 
    \|_{ 
      C([0,1],\mathbb{R}) 
    }
\end{equation}
for every $ t \in (0,T] $
and every 
$ v \in P_N( C([0,1], \mathbb{R}) ) $.
In the next step define the
continuous function
$
  y_N \colon [ 0, \tau ]
  \rightarrow
  P_N( C([0,1], \mathbb{R}) )
$
by
\begin{equation}
  y_N(t) 
  := x_N(t) - o_N(t) 
  =
  \int_{0}^{t}  
    P_N \, S_{ t - s } \, F( x_N(s) ) \, 
  ds
  =
  \int_{0}^{t}  
    S_{ t - s } \, P_N \, F( x_N(s) ) \, 
  ds
\end{equation}
for every $ t \in [0, \tau] $.
Here the $ \mathbb{R} $-vector space
$
  P_N( C([0,1], \mathbb{R}) )
$
is equipped with the
supremum norm
$
  \left\| v \right\|_V
  =
  \sup_{ 0 \leq x \leq 1 }
  | v(x) |
$
for every
$ 
  v \in 
  P_N( C([0,1], \mathbb{R}) )
$.
Furthermore, observe that
equation~\eqref{lem7proof}
implies
\begin{equation}
\label{eq:limit0}
\begin{split}
&
  \frac{ 
    y_N(t_2) - y_N(t_1) 
  }{
    t_2 - t_1
  }
= 
    \frac{ 1 }{ t_2 - t_1 } 
    \int_{t_1}^{t_2}
      S_{ t_2 - s } \, P_N \, F( x_N(s) ) \, 
    ds
  + 
  \frac{ 
    \left( S_{ t_2 - t_1 } - I
    \right)   
    y_N(t_1)
  }{ t_2 - t_1 }
\\ & =
    \frac{ 1 }{ t_2 - t_1 } 
    \int_{t_1}^{t_2}
      S_{ t_2 - s } \, P_N \, F( x_N(s) ) \, 
    ds
  + 
  \frac{ 
    1
  }{ t_2 - t_1 }
  \int_{ 0 }^{ t_2 - t_1 }
  S_s \;
  \Delta y_N(t_1) 
  \,
  ds
\end{split}
\end{equation}
for all 
$
  t_1, t_2 \in [0,\tau]
$
with 
$
  t_1 < t_2
$.
Here and below
$ 
  \Delta y_N(t) 
$
is the second derivative of
$ y_N(t) $ in the spatial variable,
i.e.,
$
  ( \Delta y_N(t) )(x)
  =
  \big(
  \frac{ \partial^2 y_N 
  }{
    \partial x^2
  }
  \big)(t,x)
$
for all 
$ t \in [0,\tau] $
and all
$ x \in [0,1] $.
Next again \eqref{lem7proof} 
implies
\begin{align}
\label{eq:limit1}
&
  \lim_{
    \substack{
      t_1 \nearrow t, \, t_2 \searrow t \\
      0 \leq t_1 < t_2 \leq \tau
    }
  }
  \frac{ 
    1
  }{ t_2 - t_1 }
  \int_{ 0 }^{ t_2 - t_1 }
  S_s \;
  \Delta y_N(t_1) 
  \,
  ds
  =
  \Delta y_N(t) 
  \qquad \text{and}
\\ 
\label{eq:limit2}
&
  \lim_{
    \substack{
      t_1 \nearrow t, \, t_2 \searrow t \\
      0 \leq t_1 < t_2 \leq \tau
    }
  }
    \frac{ 1 }{ t_2 - t_1 } 
    \int_{t_1}^{t_2}
      S_{ t_2 - s } \, P_N \, F( x_N(s) ) \, 
    ds
  =
  P_N \, F( x_N(t) ) 
\end{align}
for every $ t \in [0,\tau] $.
Combining \eqref{eq:limit0}-\eqref{eq:limit2}
then results in
\begin{equation}
  \tfrac{ \partial }{ \partial t } 
  \,
  y_N(t)
=
  \Delta y_N(t) 
  + P_N \, F( x_N(t) ) 
=
  \Delta y_N(t) 
  + 
  P_N \,
  F\big( 
    y_N(t) + o_N(t)
  \big)
\end{equation}
and Lemma~\ref{lemburger}
hence gives
\begin{align*}
  \tfrac{ \partial }{ \partial t }  
  \left\| 
    y_N(t)
  \right\|_{L^2}^2
& 
  = 2 \left< 
    y_N(t), 
    \Delta y_N(t)
    + F( y_N(t) + o_N(t) ) 
  \right>_{L^2}  
\\ & \leq 
  4 c^2 
  \| y_N(t) 
  \|^2_{L^2} \!
  \sup\limits_{ 0 \leq s \leq \tau }
  \| o_N(s) \|^2_{ C([0,1] , \mathbb{R}) }
  + 
  4 c^2 \!
    \sup_{0 \leq s \leq \tau}
    \| o_N(s) \|^4_{  
      C( [0,1], \mathbb{R}) 
    }
\end{align*}
for every $ t \in [0, \tau] $. 
Gronwall's lemma 
and the estimates
$ x^2 \leq e^x $
and
$
  x \leq e^x
$
for all $ x \in [0,\infty) $
therefore yield
\begin{equation}
\label{r1proof}
  \| y_N(t) \|^2_{ L^2 } 
\leq 
  \exp\!\Big(
    4 c^2 z^2 T
  \Big)
  4 c^2 z^4
  T
\leq
  e^{
    4 c^2 z^2 T 
    +
    2 | c | z^2
    +
    T
  }
\leq 
  e^{
    4 (c^2 + 1) ( T + 1) z^2
  }
\end{equation}
for every $ t \in [0, \tau] $ 
where here and below
$
  z :=
  \max\!\big( 
    1,
    \sup\limits_{ 
      0 \leq s \leq \tau 
    }
    \| 
      o_N(s) 
    \|_{ 
      C( [0, 1] , \mathbb{R} ) 
    }
  \big)
$.
In the next step 
Lemma~\ref{lemmadditional}
shows
\begin{equation}
\label{r2proof}
\begin{split}
  \| y_N(t)\|_{ L^4 }
  &=
  \int^t_0 \|P_N \, S_{ t - s } \, F( x_N(s) ) \|_{ L^4 } \, ds
\\ & \leq
  2^{ \frac{1}{8} }	
  |c| 
  \int^t_0 
    ( t - s )^{ - \frac{1}{8} }
    \,
    \Big\|
      S_{ \frac{ (t-s) }{ 2 } }
      \Big( 
      \left( 
        ( x_N(s) )^2 
      \right)' 
      \Big)
    \Big\|_{ L^2 }
  \, ds
\\ & \leq
  2^{ \frac{1}{8} }	
  |c|
  \int^t_0 
      ( t - s )^{ - \frac{1}{8} }
    \cdot
    4 \cdot (T+1) \cdot
    \left( \tfrac{t-s}{2} \right)^{-\frac{3}{4}}
    \cdot
    \| ( x_N(s) )^2 \|_{ L^1 }
    \,
  ds
\\ & \leq
  64 \left( T + 1 \right) 
  T^{ \frac{1}{8} }	 
  \, |c|
  \left(
    \sup_{0 \leq s \leq \tau }
    \| x_N(s)\|_{L^2}^2
  \right)
\end{split}
\end{equation}
for every $ t \in [0,\tau] $.
Additionally, again
Lemma~\ref{lemmadditional}
gives
\begin{align}
\label{r3proof}
  & 
  \|
    y_N(t) 
  \|_{ C([0,1] , \mathbb{R}) }
=
  \Big\|
    \int^t_0
      P_N \, S_{ t - s } \,
      F( x^N(s) ) \,
    ds
  \Big\|_{ C([0,1] , \mathbb{R}) }
\\ & \leq 
\nonumber
  2 
  \int^t_0
    (t-s)^{ - \frac{3}{4} } \,
    \| F( x^N(s) )   \|_{H^{-1}}
\,
  ds
  \leq 
  8 \, | c | \, T^{ \frac{1}{4} } 
  \left(
    \sup_{ 0 \leq s \leq \tau }
     \| x_N(s) \|_{ L^4 }^2
  \right)
\end{align}
for every $ t \in [0, \tau] $. 
Combining 
\eqref{r2proof} 
and \eqref{r3proof}
then yields
\begin{align*}
&
  \sup_{0 \leq t \leq \tau }
  \| y^N(t) \|_{ C([0,1] , \mathbb{R}) }
\leq 
  8 \, |c| \, T^{\frac{1}{4}}
  \Big( 
    \sup_{0 \leq t \leq \tau }
    \| y^N(t) \|_{L^4}
    + 
    \sup_{0 \leq t \leq \tau }
    \| o^N(t) \|_{L^4}
  \Big)^{ \! 2 }
\\ & \leq 
  8 \, |c| \, T^{ \frac{1}{4} } 
  \Big( 
    64 \left(T + 1\right) 
    T^{ \frac{1}{8} } 
    \, |c| 
    \sup_{0 \leq t \leq \tau }
    \| x_N(t) \|_{L^2}^2
    + z
  \Big)^{ \! 2 }
\\ & \leq 
  2^{16} \, ( |c|^3 +1 ) \,
  (T+1)^3
  \Big( 
    \Big(
      \sup_{0 \leq t \leq \tau }
      \|y_N(t) \|_{L^2} + z
    \Big)^{ \! 4 }
    + z^2
  \Big) .
\end{align*}
Inequality~\eqref{r1proof}
therefore shows
\begin{align*}
&
  \sup_{0 \leq t \leq \tau }
  \| 
    y_N(t)
  \|_{ C([0,1] , \mathbb{R}) }
  \leq 
  2^{ 19 } \, 
  ( |c|^3 +1 ) \,
  (T+1)^3
  \Big( 
    \sup_{0 \leq t \leq \tau }
    \| y_N(t) \|^4_{ L^2 } 
    + 
    2 z^4
  \Big)
\\ & \leq
  2^{ 19 } \, ( |c|^3 + 1 ) \,
  (T+1)^3
  \Big( 
    e^{ 
      8 ( c^2 + 1 ) ( T + 1 ) z^2 
    }
    + 2 z^4
  \Big) 
\\ & \leq
  \left(
    2^{ 7 } 
    \left( c^2 + 1 \right)
    \left( T + 1 \right)
  \right)^{ 3 }
  e^{ 
    8 ( c^2 + 1 ) ( T + 1 ) z^2 
  }
\leq
  e^{  
    23 ( c^2 + 1 ) ( T + 1 ) z^2
  } 
\end{align*}
and this completes the
proof of 
Lemma~\ref{lemproof}.
\end{proof}
\subsubsection{Acknowledgement}
We are very grateful to
Sebastian Becker for his 
considerable help with the 
numerical simulations
and his perfect typing job.
We also thank an anonymous 
referee for his helpful remarks,
particularly, for pointing out 
a useful generalization in 
Assumption~\ref{semigroup}
to us.

\bibliographystyle{siam}
\bibliography{bibfile}

\end{document}